\documentclass[10pt]{amsart}
\usepackage{amssymb}
\usepackage{dsfont}
\usepackage{amscd}
\usepackage[mathscr]{euscript} 
\usepackage[all]{xy}
\usepackage{url}
\usepackage{stmaryrd}
\usepackage{comment}
\usepackage[retainorgcmds]{IEEEtrantools}
\usepackage{hyperref}
\usepackage{graphicx}
\usepackage{mathtools}
\usepackage{color}
\usepackage{centernot}
\usepackage{marvosym}
\usepackage{tikz}
\usepackage{arydshln}
\usepackage{float}
\usepackage{faktor}

\title{On Galois groups and PAC substructures}
\author[D. M. HOFFMANN]{Daniel Max Hoffmann$^{\dagger}$}
\thanks{$^{\dagger}$SDG. Supported by NCN (National Science Centre, Poland) grants no. 2016/21/N/ST1/01465,
and 2015/19/B/ST1/01150.}
\address{$^{\dagger}$Instytut Matematyki\\
Uniwersytet Warszawski\\
Warszawa\\
Poland}
\email{daniel.max.hoffmann@gmail.com}
\urladdr{https://www.researchgate.net/profile/Daniel\_Hoffmann8}

\DeclareMathOperator{\fratt}{Fratt}

\DeclareMathOperator{\acl}{acl} \DeclareMathOperator{\dcl}{dcl} 
 \DeclareMathOperator{\aut}{Aut} \DeclareMathOperator{\id}{id}

\DeclareMathOperator{\stab}{Stab}

\DeclareMathOperator{\imG}{Im}

 \DeclareMathOperator{\theo}{Th}

 \DeclareMathOperator{\eq}{eq}
 
\DeclareMathOperator{\tp}{tp}

\DeclareMathOperator{\ddf}{DF}\DeclareMathOperator{\dcf}{DCF}\DeclareMathOperator{\scf}{SCF}

\DeclareMathOperator{\qftp}{qftp}

\newtheorem{theorem}{Theorem}[section]
\newtheorem{prop}[theorem]{Proposition}
\newtheorem{lemma}[theorem]{Lemma}
\newtheorem{cor}[theorem]{Corollary}
\newtheorem{fact}[theorem]{Fact}
\theoremstyle{definition}
\newtheorem{definition}[theorem]{Definition}
\newtheorem{example}[theorem]{Example}
\newtheorem{remark}[theorem]{Remark}
\newtheorem{question}[theorem]{Question}
\newtheorem{conj}[theorem]{Conjecture}

\theoremstyle{remark}
\newtheorem*{theorem*}{Theorem}
\newtheorem*{cor*}{Corollary}

\theoremstyle{definition}
\theoremstyle{definition}

\theoremstyle{definition}

\theoremstyle{remark}

\makeatletter
\providecommand*{\cupdot}{%
  \mathbin{%
    \mathpalette\@cupdot{}%
  }%
}
\newcommand*{\@cupdot}[2]{%
  \ooalign{%
    $\m@th#1\cup$\cr
    \sbox0{$#1\cup$}%
    \dimen@=\ht0 %
    \sbox0{$\m@th#1\cdot$}%
    \advance\dimen@ by -\ht0 %
    \dimen@=.5\dimen@
    \hidewidth\raise\dimen@\box0\hidewidth
  }%
}

\providecommand*{\bigcupdot}{%
  \mathop{%
    \vphantom{\bigcup}%
    \mathpalette\@bigcupdot{}%
  }%
}
\newcommand*{\@bigcupdot}[2]{%
  \ooalign{%
    $\m@th#1\bigcup$\cr
    \sbox0{$#1\bigcup$}%
    \dimen@=\ht0 %
    \advance\dimen@ by -\dp0 %
    \sbox0{\scalebox{2}{$\m@th#1\cdot$}}%
    \advance\dimen@ by -\ht0 %
    \dimen@=.5\dimen@
    \hidewidth\raise\dimen@\box0\hidewidth
  }%
}
\makeatother

\def\Ind#1#2{#1\setbox0=\hbox{$#1x$}\kern\wd0\hbox to 0pt{\hss$#1\mid$\hss}
\lower.9\ht0\hbox to 0pt{\hss$#1\smile$\hss}\kern\wd0}

\def\ind{\mathop{\mathpalette\Ind{}}}

\def\notind#1#2{#1\setbox0=\hbox{$#1x$}\kern\wd0
\hbox to 0pt{\mathchardef\nn=12854\hss$#1\nn$\kern1.4\wd0\hss}
\hbox to 0pt{\hss$#1\mid$\hss}\lower.9\ht0 \hbox to 0pt{\hss$#1\smile$\hss}\kern\wd0}

\begin{document}

\newcommand{\twoc}[3]{ {#1} \choose {{#2}|{#3}}}
\newcommand{\thrc}[4]{ {#1} \choose {{#2}|{#3}|{#4}}}
\newcommand{\Rr}{{\mathds{R}}}
\newcommand{\Kk}{{\mathds{K}}}

\newcommand{\dlog}{\mathrm{ld}}
\newcommand{\ga}{\mathbb{G}_{\rm{a}}}
\newcommand{\gm}{\mathbb{G}_{\rm{m}}}
\newcommand{\gaf}{\widehat{\mathbb{G}}_{\rm{a}}}
\newcommand{\gmf}{\widehat{\mathbb{G}}_{\rm{m}}}
\newcommand{\gdf}{\mathfrak{g}-\ddf}
\newcommand{\gdcf}{\mathfrak{g}-\dcf}
\newcommand{\fdf}{F-\ddf}
\newcommand{\fdcf}{F-\dcf}
\newcommand{\mw}{\scf_{\text{MW},e}}

\begin{abstract}
\begin{enumerate}
\item
We show that for an arbitrary stable theory $T$, a group $G$ is profinite if and only if $G$ occurs as a Galois group of some Galois extension inside a monster model of $T$.

\item
We prove that any PAC substructure of the monster model of $T$ has projective absolute Galois group. 

\item
Moreover, any projective profinite group $G$ is isomorphic to the absolute Galois group of a definably closed substructure $P$ of the monster model.
If $T$ is $\omega$-stable, then $P$ can be chosen to be PAC.

\item
Finally, we provide a description of some Galois groups of existentially closed substructures with $G$-action in the terms of the universal Frattini cover. Such structures might be understood as a new source of examples of PAC structures.
\end{enumerate}
\end{abstract}
\maketitle

\section{Introduction}
Content of this paper is a part of our project on \emph{pseudo-algebraically closed} (PAC) structures. The project aims to generalize the Galois theory from the case of PAC fields to the case of PAC structures and to use a description of PAC structures to provide new examples in the studies on NSOP$_1$. Due to the Elementarily Equivalence Theorem for PAC fields (Theorem 20.3.3 in \cite{FrJa}),
PAC fields were the core of research in the field theory in the second half of the twentieth century.
In a subsequent paper (\cite{DHL1}), we achieve a desired generalization of the Elementarily Equivalence Theorem and therefore this direction of research seems to be promising.
Many arguments in the proofs of the more model-theoretical part of the research on PAC fields are based on the fact that the absolute Galois group of a PAC field is projective (e.g. a characterization of the model-theoretic algebraic closure and the forking independence in a PAC field, see \cite{ChaPil}). In Theorem \ref{PAC.proj}, we provide a suitable generalization of the aforementioned fact: the absolute Galois groups of a PAC structure is projective. Moreover, in the case of PAC fields, the absolute Galois group plays crucial role. In Section \ref{sec.G.actions}, we describe a method for obtaining PAC structures with a desired absolute Galois group, so ``controlled" PAC structures:

\begin{figure}[H]
\begin{tabular}{r|l}
 fact & status \\ \hline
e.c. structures &\\ with $G$-action are PAC& proven in \cite{Hoff3} \\ \hdashline
 PAC structures are &\\
 controlled by Galois groups & proven in \cite{DHL1} \\ \hdashline
 description of Galois groups &\\
 of e.c. structures with $G$-action & see Corollary \ref{cor.Galois.descr}
\end{tabular}
\end{figure}

In a broader context,
studying PAC substructures was, and still is, interesting as a generalization of a very elegant part of the theory of fields. Important results were achieved in the case of strongly minimal theories (in \cite{manuscript}) and then bounded PAC structures were described in the case of a stable ``ambient overstructure" as structures with simple theory (\cite{PilPol} and \cite{Polkowska}). 
Among the other 
model-theoretical results on PAC fields 
(e.g. \cite{ershov1980}, \cite{CDM81}, \cite{cherlindriesmacintyre}, \cite{ChaPil}, \cite{chatzidakis2002}, 
\cite{chahru04}) there are many results promising a very interesting generalization to the level of PAC structures.
For example, Nick Ramsey in his doctoral dissertation (\cite{nickPhD}) shows that a PAC field is NSOP$_1$ if and only if its absolute Galois group is NSOP$_1$ (as a special many sorted structure).
Our next goal is to generalize Ramsey's result to the level of PAC structures and use it, with our previous results, to provide examples of non-field PAC structures which are NSOP$_1$ structures, and which can be somehow controlled.

We expose our results in a connection with generalizations of some well known facts from the classical Galois theory.
Suppose that $T$ is a stable complete $\mathcal{L}$-theory with quantifier elimination and elimination of imaginaries (possibly many sorted), $\mathfrak{C}$ is a monster model of $T$, and $G$ is a group smaller than the saturation of $\mathfrak{C}$.
In this paper, we are interested in the following two properties related to profinite groups:
\begin{enumerate}
\item[(P$_T$)] $G$ is profinite iff there exists a Galois extension $A\subseteq B$ in $\mathfrak{C}$ such that $G\cong\aut_{\mathcal{L}}(B/A)$
\item[(PP$_T$)] profinite $G$ is projective iff $G\cong\mathcal{G}(P)$ for some PAC substructure $P$ of $\mathfrak{C}$
\end{enumerate}
(where $\mathcal{G}(P)$ is the absolute Galois group of $P$, for definitions of the rest of notions appearing in the above lines check Section \ref{sec.basics}). By well known facts, properties P$_T$ and PP$_T$ hold for $T=$ACF (e.g. Fact \ref{fact317}, Corollary 1.3.4 and Corollary 23.1.3 in \cite{FrJa}). We ask here whether other stable theories enjoy properties P$_T$ and PP$_T$.
By Corollary \ref{pro.embed2} Property P$_T$ holds, which 
was not hard to show, for any stable theory $T$. On the other hand, we ``encountered" difficulties with proving property PP$_T$ for any stable theory $T$ (with elimination of quantifiers and elimination of imaginaries). 
Theorem \ref{fr.ja.2311} shows that any projective profinite group $G$ is isomorphic to the absolute Galois group of some (definably closed) substructure $P$. However, to state that $P$ is PAC, additional assumptions on $T$ are needed (the ``moreover" part in Theorem \ref{fr.ja.2311}) - we need to assume that any type over a small $A\subseteq\mathfrak{C}$ has only finitely many extensions over $\acl_{\mathcal{L}}^{\mathfrak{C}}(A)$. This additional assumption is related to Lemma \ref{alg.ext.PAC}, which is used in the proof of the ``moreover" part of Theorem \ref{fr.ja.2311}. It turns out that, in opposite to the theory of fields, an algebraic extension of a PAC substructure is not necessary PAC (see Remark \ref{counterexamples.alg.PAC}) and that was the main obstacle in proving that any stable theory $T$ enjoys property PP$_T$. 

The central result of this paper is Theorem \ref{PAC.proj}. It was already known for fields and also in the case of strongly minimal theories (Lemma 1.17 in \cite{manuscript}). We used the notion of regularity to stretch the proof of Lemma 1.17 in \cite{manuscript} over the case of stable theories.

Section \ref{sec.G.actions} uses previous results to describe Galois groups of substructures of our stable monster model $\mathfrak{C}$, which are equipped with a group action by automorphisms. In \cite{Hoff3}, we proved that existentially closed substructures of $\mathfrak{C}$ equipped with a group action are PAC substructures, similarly as invariants of the group action. Hence consideration of 
existentially closed substructures with a group action provides a new class of PAC substructures for any stable theory $T$. It was natural to check what we can say about absolute Galois groups of such PAC substructures, and so get an intuition about elementary invariants of PAC substructures in general. Some statements in Section \ref{sec.G.actions} generalize theorems from \cite{sjogren}, which focus on the theory of fields with a group action. More details might be found in Section \ref{sec.G.actions}.

We thank Anand Pillay, Ludomir Newelski and Nick Ramsey for helpful discussions which clarified several problems occurring during our work on this paper.


\section{Basics}\label{sec.basics}
\subsection{Preliminaries and conventions}
If $A$ and $B$ are two sequences, then $AB$ denotes the concatenation of $A$ and $B$, i.e. $A^{\frown}B$. If $A$ and $B$ are considered only as sets, then $AB$ denotes $A\cup B$. Finally, if $H$ is a group and $A$ is a set, then the orbit of $A$ under an action of $H$ will be denoted by $H\cdot A$ or (if it will not lead to any confusion) by $HA$.

We assume that theories in this paper are theories with infinite models.
Let $N$ and $N'$ be $\mathcal{L}$-structures and let $E$ be a subset of $N$. We use $\langle E\rangle_{\mathcal{L}}$ to denote the $\mathcal{L}$-substructure of $N$ generated by $E$. Moreover, $\acl_{\mathcal{L}}^N(E)$ denotes the algebraic closure of $E$ in $N$ in the sense of the language $\mathcal{L}$ and the $\mathcal{L}$-theory $\theo(N)$ (similarly for $\dcl_{\mathcal{L}}^{N}(E)$ and $\tp_{\mathcal{L}}^N(a/E)$). 

We fix (``once for all") an $\mathcal{L}$-structure $\mathfrak{C}$ which is $\kappa$-saturated and $\kappa$-strongly homogeneous and set $T:=\theo_{\mathcal{L}}(\mathfrak{C})$. In other words: $\mathfrak{C}$ is a monster model for the complete theory $T$. At some point, we will start to assume that $T$ is stable and has quantifier elimination and elimination of imaginaries, which is rather a technical assumption if $T$ is stable, since we are able to force every stable $T$ to have quantifier elimination and elimination of imaginaries by passing to $(T^{\eq})$ and then taking the Morleyisation $(T^{\eq})^m$ of the theory $T^{\eq}$.

A group $G$ will be considered in different places of the paper, but we \textbf{always assume that $|G|<\kappa$}. 

We use in the whole text the following convention:
results named by ``Fact" are things recalled from previous papers, results named in other way are generalizations or new obtained results.

\subsection{Old definitions and facts}
In this subsection, we provide notions and facts from \cite{Hoff3}, which are the basis for the rest of the paper. We omitted proofs, but always indicated corresponding fact in \cite{Hoff3}, so a reader interested in proofs can easily check them in \cite{Hoff3}.

We recall here the notion of a \emph{regular extension}, a \emph{PAC substructure}, and a \emph{Galois extension}.

\begin{definition}\label{regular.def}
\begin{enumerate}
\item Let $E\subseteq A$ be small subsets of $\mathfrak{C}$. We say that $E\subseteq A$ is \emph{$\mathcal{L}$-regular} (or just \emph{regular}) if
$$\dcl_{\mathcal{L}}^{\mathfrak{C}}(A)\cap\acl_{\mathcal{L}}^{\mathfrak{C}}(E)=\dcl_{\mathcal{L}}^{\mathfrak{C}}(E).$$

\item Let $N$ be a small $\mathcal{L}$-substructure of $\mathfrak{C}$. We say that $N$ is \emph{pseudo-algebraically closed} (\emph{PAC}) if for every small $\mathcal{L}$-substructure $N'$ of $\mathfrak{C}$, which is $\mathcal{L}$-regular extension of $N$, it follows $N\preceq_1 N'$ (i.e. $N$ is existentially closed in $N'$).
\end{enumerate}
\end{definition}
\noindent

\begin{fact}[Remark 3.2 in \cite{Hoff3}]\label{regular.remark}
\begin{enumerate}
\item
Note that the regularity condition is invariant under the action of automorphisms.

\item
Of course, if $E$ is algebraically closed, then $E\subseteq A$ is regular for any small $A$.

\item
If $E\subseteq A$ is regular, and $E\subseteq A'\subseteq A$, then $E\subseteq A'$ is regular.

\item 
Assume that $E\subseteq A$ and $A\subseteq B$ are regular. It follows that $E\subseteq B$ is regular.

\item 
Let $P$ be a small $\mathcal{L}$-substructure of $\mathfrak{C}$. There exists
a small $\mathcal{L}$-substructure $P^{\ast}$ of $\mathfrak{C}$ such that $P\subseteq P^{\ast}$ is regular and $P^{\ast}$ is PAC.
\end{enumerate}
\end{fact}

From this point we assume that $T$ allows to \textbf{eliminate quantifiers}. Note that quantifier elimination in $T$ implies that for a small PAC substructure $P\subseteq\mathfrak{C}$ it follows $\dcl_{\mathcal{L}}^{\mathfrak{C}}(P)=P$.

\begin{fact}[Lemma 3.3 in \cite{Hoff3}]\label{lang410}
If for some small $\mathcal{L}$-substructures $P\subseteq N$ of $\mathfrak{C}$ it is $P\preceq_1 N$, then $P\subseteq N$ is regular.
\end{fact}

\begin{definition}\label{galois.ext.def}
\begin{enumerate}
\item Assume that $A\subseteq C$ are $\mathcal{L}$-substructures of $\mathfrak{C}$. We say that $C$ is \emph{normal over $A$} (or we say that $A\subseteq C$ is a \emph{normal extension}) if $\aut_{\mathcal{L}}(\mathfrak{C}/A)\cdot C\subseteq C$.
(Note that if $C$ is small and $A\subseteq C$ is normal, then it must be $C\subseteq\acl_{\mathcal{L}}^{\mathfrak{C}}(A)$.)

\item Assume that $A\subseteq C\subseteq\acl_{\mathcal{L}}^{\mathfrak{C}}(A)$ are small $\mathcal{L}$-substructures of $\mathfrak{C}$ such that $A=\dcl_{\mathcal{L}}^{\mathfrak{C}}(A)$, $C=\dcl_{\mathcal{L}}^{\mathfrak{C}}(C)$ and $C$ is normal over $A$. In this situation we say that $A\subseteq C$ is a \emph{Galois extension}.
\end{enumerate}
\end{definition}

We evoke here several facts about Galois extensions, which will be used in the rest of the paper. These facts are standard, so the reader can treat them as exercises to the above definition of a Galois extension.

\begin{fact}[Corollary 7. in \cite{invitation}]\label{fact.exact}
Let $A$, $B$ and $C$ be small $\mathcal{L}$-substructures of $\mathfrak{C}$ such that $A\subseteq B\subseteq C\subseteq \acl_{\mathcal{L}}^{\mathfrak{C}}(A)$, $C$ and $B$ are normal over $A$. Then
$$\xymatrix{1 \ar[r] & \aut_{\mathcal{L}}(C/B) \ar[r]^{\subseteq} & \aut_{\mathcal{L}}(C/A) \ar[r]^{|_B} & \aut_{\mathcal{L}}(B/A) \ar[r] & 1}$$
is an exact sequence and hence $\aut_{\mathcal{L}}(C/B)\trianglelefteqslant \aut_{\mathcal{L}}(C/A)$.
\end{fact}

\begin{fact}[Fact 3.20 in \cite{Hoff3}]\label{fact314}
Assume that $A\subseteq C$ is a Galois extension and $A\subseteq B=\dcl_{\mathcal{L}}^{\mathfrak{C}}(B)\subseteq C$. The extension $A\subseteq B$ is Galois if and only if
$$\aut_{\mathcal{L}}(C/B)\trianglelefteqslant \aut_{\mathcal{L}}(C/A).$$
\end{fact}

From now we assume that $T$ additionally admits \textbf{elimination of imaginaries}.

\begin{fact}[The Galois correspondence]\label{galois.correspondence}
Let $A\subseteq C$ be a Galois extension, introduce
$$\mathcal{B}:=\lbrace B\;|\;A\subseteq B=\dcl_{\mathcal{L}}^{\mathfrak{C}}(B)\subseteq C\rbrace,$$
$$\mathcal{H}:=\lbrace H\;|\; H\leqslant\aut_{\mathcal{L}}(C/A)\text{ is closed}\rbrace.$$
Then $\alpha(B):=\aut_{\mathcal{L}}(C/B)$ is a mapping between $\mathcal{B}$ and $\mathcal{H}$, $\beta(H):=C^H$ is a mapping between $\mathcal{H}$ and $\mathcal{B}$ and it follows
$$\alpha\circ\beta=\id,\qquad\beta\circ\alpha=\id.$$
\end{fact}

\begin{fact}[Fact 3.23 in \cite{Hoff3}]\label{fact317}
If $A\subseteq C$ is a Galois extension, then $\aut_{\mathcal{L}}(C/A)$ is a profinite group.
\end{fact}

\begin{definition}
For a small subset $A$ of $\mathfrak{C}$ we define the \emph{absolute Galois group} of $A$:
$$\mathcal{G}(A):=\aut_{\mathcal{L}}\big(\acl_{\mathcal{L}}^{\mathfrak{C}}(A)/\dcl_{\mathcal{L}}^{\mathfrak{C}}(A) \big).$$
\end{definition}

The following lemma is a smooth generalization of Lemma 3.24 in \cite{Hoff3}. The original proof of \cite[Lemma 3.24]{Hoff3} still works well for items (1)-(3).

\begin{lemma}\label{N_Galois}
Assume that $N$ is a small definably closed $\mathcal{L}$-substructure of $\mathfrak{C}$ equipped with a $G$-action $(\tau_g)_{g\in G}$. 
Let $i:G\to\aut_{\mathcal{L}}(N/N^G)$ be given by $i(g):=\tau_g$.
\begin{enumerate}
\item If for every $b\in N$ the orbit $G\cdot b$ is definable, then $N^G\subseteq N$ is normal.

\item
If $N\subseteq\acl_{\mathcal{L}}^{\mathfrak{C}}(N^G)$, then $N^G\subseteq N$ is a Galois extension.

\item
If $G$ is finite, then $N\subseteq\acl_{\mathcal{L}}^{\mathfrak{C}}(N^G)$, hence also the second point follows.

\item (Artin's theorem)
If $G$ is finite and the $G$-action $(\tau_g)_{g\in G}$ is faithful, then 
$i:G \cong Aut_{\mathcal{L}}(N/N^G)$.

\item
Assume that $G$ is profinite, the $G$-action $(\tau_g)_{g\in G}$ is faithful and for every $m\in N$ the stabiliser $\stab(m)=\lbrace g\in G\;|\;\tau_g(m)=m\rbrace$ is an open subgroup of $G$. Then $N^G\subseteq N$ is Galois and
$i: G\cong Aut_{\mathcal{L}}(N/N^G)$ (as profinite groups).
\end{enumerate}
\end{lemma}

\begin{proof}
We only need to prove items (4) and (5). We start with the proof of the item (4). Since the group $G$ acts faithfully, we can embed $G$ into $Aut_{\mathcal{L}}(N/N^G)$. By item (3) we know that $N^G\subseteq N$ is Galois. By Fact \ref{fact317}, the group $Aut_{\mathcal{L}}(N/N^G)$ is profinite. Therefore the image of $G$, $i(G)$, which is finite, is a closed subgroup. We have 
$$Aut_{\mathcal{L}}(N/N^G)=N^G=Aut_{\mathcal{L}}(N/N^{i(G)}),$$
hence, by Fact \ref{galois.correspondence}, it follows that $Aut_{\mathcal{L}}(N/N^G)=i(G)\cong G$.

We move now to the proof of the item (5). Our proof is based on the proof of Lemma 1.3.2 in \cite{FrJa}, which is the same result, but for the theory $T=$ACF. Our assumptions assure us that $i:G\to\aut_{\mathcal{L}}(N/N^G)$ is an embedding of groups. We need to show that $i$ is continuous and onto, and that $N^G\subseteq N$ is Galois.

Let $m_1,\ldots,m_n\in N$ be arbitrary. The subgroup $H=\stab(m_1)\cap\ldots\cap \stab(m_n)$ is open in $G$, hence $\mathcal{N}$, equal to the intersection of all conjugates of $H$ in $G$, is open (and therefore closed and of finite index). Consider a $G/\mathcal{N}$-action on $N_0:=\dcl_{\mathcal{L}}^{\mathfrak{C}}(N^G,G\cdot m_1,\ldots,G\cdot m_n)$ given by $g(a)=i(g)|_{N_0}(a)$. It is faithful and $N_0^{G/\mathcal{N}}=N^G$. Note that $G/\mathcal{N}$ is finite, hence we obtain that $N^G\subseteq N_0$ is Galois and $G/\mathcal{N}\cong\aut_{\mathcal{L}}(N_0/N^G)$, where the isomorphism is given by $g\mathcal{N}\mapsto i(g)|_{N_0}$.

Since each $m\in N$ belongs to some $N_0$, we conclude that $N^G\subseteq N$ is
normal and $N\subseteq \acl_{\mathcal{L}}^{\mathfrak{C}}(N^G)$. Oviously $\dcl_{\mathcal{L}}^{\mathfrak{C}}(N^G)=N^G$, hence $N^G\subseteq N$ is Galois.

Now, we will show that $i$ is an isomorphism of profinite groups. Let $\{N_{\alpha}\;|\;\alpha<\beta\}$ be the set of all finite Galois extensions of $N^G$, which are of the form of $N_0$. The collection of finite groups $\aut_{\mathcal{L}}(N_{\alpha}/N^G)$ with restriction maps form an inverse system, therefore we can speak about its limit, $\big(\lim\limits_{\leftarrow_{\alpha}}\aut_{\mathcal{L}}(N_{\alpha}/N^G),\;\pi_{\alpha}\big)$. Note that
$$f:=\lim\limits_{\leftarrow_{\alpha}}f_{\alpha}:\aut_{\mathcal{L}}(N/N^G)\to\lim\limits_{\leftarrow_{\alpha}}\aut_{\mathcal{L}}(N_{\alpha}/N^G),$$
where $f_{\alpha}:\aut_{\mathcal{L}}(N/N^G)\to\aut_{\mathcal{L}}(N_{\alpha}/N^G)$, $f_{\alpha}(\sigma):=\sigma|_{N_{\alpha}}$, is a continuous isomorphism of groups (it is onto by Corollary 1.1.6 in \cite{ribzal}, it is one-to-one, because family $N_{\alpha}$ covers the whole $N$), hence it is also an homeomorphism and an isomorphism of profinite groups.

On the other hand, 
$$h:=\lim\limits_{\leftarrow_{\alpha}}h_{\alpha}:G\to\lim\limits_{\leftarrow_{\alpha}}\aut_{\mathcal{L}}(N_{\alpha}/N^G),$$
where $h_{\alpha}:G\to\aut(N_{\alpha}/N^G)$, $h_{\alpha}(g):=i(g)|_{N_{\alpha}}$, is (by Corollary 1.1.6 in \cite{ribzal} and the previous part of the proof of this point) a continuous epimorphism of groups. Hence $f^{-1}h:G\to\aut_{\mathcal{L}}(N/N^G)$ is a continuous epimorphism of groups. To finish the proof, observe that for each $g\in G$ we have $(f^{-1}h)(g)=i(g)$.
\end{proof}

Now, we add one more, but a stronger assumption: \textbf{$T$ is stable}.
The following facts help in a better understanding what exactly, in the terms of ``standard" model theory, is regularity. We will see in a moment that the notion of regularity is an algebraic way to express stationarity (check Remark 3.2.(1) in \cite{Hoff3} or Definition 5.17 in \cite{silvain002}). Moreover, in Lemma \ref{regular.Galois} we bind together regularity/stationarity and surjectivity of a restriction map on the level of absolute Galois groups.

\begin{fact}[Fact 3.34 in \cite{Hoff3}]\label{regular}
Let $E,A\subseteq\mathfrak{C}$, $A$ be $\mathcal{L}$-regular over $E$, $f_1,f_2\in\aut_{\mathcal{L}}(\mathfrak{C})$ and let $f_1|_E=f_2|_E$. Then there exists $h\in\aut_{\mathcal{L}}(\mathfrak{C})$ such that $h|_A=f_1|_A$ and $h|_{\acl_{\mathcal{L}}^{\mathfrak{C}}(E)}=f_2|_{\acl_{\mathcal{L}}^{\mathfrak{C}}(E)}$.
\end{fact}


\begin{fact}[Lemma 3.36 in \cite{Hoff3}]\label{PACclaim}
For a small set $E\subseteq\mathfrak{C}$ and a complete type $p$ over $E$ it follows:
\begin{IEEEeqnarray*}{rCl}
p\text{ is stationary} &\iff & (\forall A_0\models p)(E\subseteq EA_0\text{ is $\mathcal{L}$-regular}) \\
 &\iff & (\exists A_0\models p)(E\subseteq EA_0\text{ is $\mathcal{L}$-regular}).
\end{IEEEeqnarray*}
\end{fact}

The following fact is crucial for upcoming proofs. We do not know whether there is a well known standard model theoretic argument for it, but the proof in \cite{Hoff3} makes use of the notion of regularity.

\begin{fact}[Corollary 3.37 in \cite{Hoff3}]\label{cor.stationary_types_exist}
For every small $\mathcal{L}$-substructure $N$ of $\mathfrak{C}$ and every $n<\omega$, there exists a non-algebraic stationary type over $N$ in $n$ many variables.
\end{fact}


\begin{fact}[Corollary 3.39 in \cite{Hoff3}]\label{regular.PAPA}
Assume that $E,A,B\subseteq\mathfrak{C}$, $A$ is $\mathcal{L}$-regular over $E$, $f_1,f_2\in\aut_{\mathcal{L}}(\mathfrak{C})$, $f_1|_E=f_2|_E$. If $A\ind^{\mathfrak{C}}_E B$ and $f_1(A)\ind^{\mathfrak{C}}_{f_1(E)} f_2(B)$, then there exists $h\in\aut_{\mathcal{L}}(\mathfrak{C})$ such that $h|_A=f_1|_A$ and $h|_B=f_2|_B$.
\end{fact}

Now, we use the above fact to express regularity in the terms of absolute Galois groups.

\begin{lemma}\label{regular.Galois}
Assume that $N\subseteq N'$ are small subsets of $\mathfrak{C}$. The set $N'$ is regular over $N$ if and only if the restriction map $\mathcal{G}(N')\to\mathcal{G}(N)$ is onto.
\end{lemma}

\begin{proof}
If $N'$ is regular over $N$, then $\dcl_{\mathcal{L}}^{\mathfrak{C}}(N)\subseteq \dcl_{\mathcal{L}}^{\mathfrak{C}}(N')$ is a regular extension and we can use Fact \ref{regular.PAPA} to show the surjectivity.

Let $n\in\dcl_{\mathcal{L}}^{\mathfrak{C}}(N')\cap\acl_{\mathcal{L}}^{\mathfrak{C}}(N)$ and let $f\in\aut_{\mathcal{L}}(\mathfrak{C}/N)$. The map $f|_{\acl_{\mathcal{L}}^{\mathfrak{C}}(N)}$ belongs to $\aut_{\mathcal{L}}\big(\acl_{\mathcal{L}}^{\mathfrak{C}}(N)/N\big)=\mathcal{G}(N)$ and therefore it is a restriction of some $\tilde{f}\in\mathcal{G}(N')$. Because $n\in\dcl_{\mathcal{L}}^{\mathfrak{C}}(N')$, we have $f(n)=\tilde{f}(n)=n$, thus $n\in\dcl_{\mathcal{L}}^{\mathfrak{C}}(N)$.
\end{proof}

\begin{fact}[Lemma 3.40 in \cite{Hoff3}]\label{lang413}
Assume that $E\subseteq A$ is $\mathcal{L}$-regular, $E\subseteq B$ and $B\ind_E^{\mathfrak{C}} A$,
then $B\subseteq BA$ is $\mathcal{L}$-regular.
\end{fact}

\begin{fact}[Corollary 3.41 in \cite{Hoff3}]\label{cor.413}
Assume that $E\subseteq A$ and $E\subseteq B$ are $\mathcal{L}$-regular, and $B\ind_E^{\mathfrak{C}} A$,
then $E\subseteq BA$ is $\mathcal{L}$-regular.
\end{fact}

\section{Profinite group as Galois group}
Now, we will make use of Lemma \ref{N_Galois}, to show that every profinite group is isomorphic to some Galois group present in our stable structure $\mathfrak{C}$. The following proposition is a straightforward generalization of Proposition 1.3.3 in \cite{FrJa}.

\begin{prop}\label{fr.ja.133}
Assume that $N_0\subseteq N$ is a Galois extension of small substructures of $\mathfrak{C}$ and 
that there is an epimorphism of profinite groups $\alpha: G\to\aut_{\mathcal{L}}(N/N_0)$.
There exist small substructures $M_0,M$ of $\mathfrak{C}$, such that $N_0\subseteq M_0$ and $N\subseteq M$ are regular, $M_0\subseteq M$ is Galois, and 
there is an isomorphism $\beta:G\to\aut_{\mathcal{L}}(M/M_0)$ such that 
$$\xymatrix{G \ar[r]^-{\beta} \ar[dr]_-{\alpha} & \aut_{\mathcal{L}}(M/M_0) \ar[d]^{|_N}\\ & \aut_{\mathcal{L}}(N/N_0)
}$$
is commuting.
\end{prop}

\begin{proof}
Set $X$ denotes the disjoint union of all quotient groups $G/\mathcal{N}$, where $\mathcal{N}$ is an open normal subgroup of $G$. Assume that $X$ is ordered in some way, say $X=\lbrace x_{\lambda}\;|\;\lambda<\lambda'\rbrace$.

Consider a non-algebraic stationary type $p(x)\in S(N_0)$ (Fact \ref{cor.stationary_types_exist}) and a Morley sequence in $p(x)$ indexed by the set $X$, $\bar{b}=(b_{g\mathcal{N}})_{g\mathcal{N}\in X}$. 

We define a $G$-action on the substructure $\dcl_{\mathcal{L}}^{\mathfrak{C}}(N_0\bar{b})$ in the following way
$$g'\cdot b_{g\mathcal{N}}:=b_{g'g\mathcal{N}}$$
(similarly as in the proof of Proposition 3.57 in \cite{Hoff3}, we inductively prove that there exists an automorphism of $\mathfrak{C}$ satisfying the above line).

By Fact \ref{cor.413}, a transfinite induction shows that $N_0\subseteq\dcl_{\mathcal{L}}^{\mathfrak{C}}(N_0\bar{b})$ is regular. 
Since $N$ is an algebraic over $N_0$, we can use Fact \ref{regular} to extend the above defined $G$-action on $\dcl_{\mathcal{L}}^{\mathfrak{C}}(N_0\bar{b})$ and $G$-action on $N$ given by
$$g\cdot m=\alpha(g)(m),$$
where $g\in G$ and $m\in N$, to a $G$-action on $M:=\dcl_{\mathcal{L}}^{\mathfrak{C}}(N\bar{b})$, say $\beta:G\to\aut_{\mathcal{L}}(M)$.
By Fact \ref{lang413}, it follows that $N\subseteq M$ is regular.

Let $M_0$ denote $M^G$.
Note that $M_0\cap N=N_0$ . To see this take $m\in M_0\cap N$. Because $m\in M_0$, it follows $m=g\cdot m =\alpha(g)(m)$ for each $g\in G$. Since $\alpha$ is onto, we obtain $m=f(m)$ for each $f\in\aut_{\mathcal{L}}(N/N_0)$, hence $m\in N_0$.

To show that $N_0\subseteq M_0$ is regular, recall that $N\subseteq M$ is regular:
$$M\cap\acl_{\mathcal{L}}^{\mathfrak{C}}(N)=N.$$
After intersecting both sides with $M_0$, we get
$$M_0\cap\acl_{\mathcal{L}}^{\mathfrak{C}}(N_0)=M_0\cap N=N_0.$$

Note that $\stab(b_{g\mathcal{N}})=\mathcal{N}$, which is an open subgroup of $G$, and for $m\in N$ we have $\stab(m)=\alpha^{-1}\Big(\lbrace f\in\aut_{\mathcal{L}}(N/N_0)\;|\; f(a)=a \rbrace\Big)$, which also is an open subgroup of $G$. We see that the action of group $G$ on $M$ is faithful. Before we can use Lemma \ref{N_Galois}.(5), we need to check whether for every $m\in M$ the stabilizer $\stab(m)$ is an open subgroup of $G$.

Because $m\in M$, there exists an $\mathcal{L}$-formula $\psi$ such that for some $a_1,\ldots,a_n\in N$ and some $g_1\mathcal{N}_1,\ldots,g_{n'}\mathcal{N}_{n'}\in X$
we have
$$\psi(a_1,\ldots a_n,b_{g_1\mathcal{N}_1},\ldots,b_{g_{n'}\mathcal{N}_{n'}},\mathfrak{C})=\lbrace m\rbrace.$$
Therefore $\stab(m)$ contains open subgroup
$$\stab(a_1)\cap\ldots\cap\stab(a_n)\cap\stab(b_{g_1\mathcal{N}_1})\cap\ldots\cap \stab(b_{g_{n'}\mathcal{N}_{n'}}),$$
hence $\stab(m)$ is an open subgroup of $G$.
By Lemma \ref{N_Galois}.(5) it follows that $M_0\subseteq M$ is Galois 
and $\beta:G\to\aut_{\mathcal{L}}(M/M_0)$ is an isomorphism of profinite groups.
The last thing we need to check is that for any $g\in G$ it follows that
$\beta(g)|_N=\alpha(g)$, but 
it follows from the construction of the $G$-action on $M$.
\end{proof}

\begin{cor}\label{pro.embed}
For every profinite group $G$ there exist a Galois extension $M_0\subseteq M$ of small substructures of $\mathfrak{C}$ such that $G\cong\aut_{\mathcal{L}}(M/M_0)$.
\end{cor}

\begin{cor}\label{pro.embed2}
A group $G$ is profinite if and only if there exist a Galois extension $M_0\subseteq M$ of small substructures of $\mathfrak{C}$ such that $G\cong\aut_{\mathcal{L}}(M/M_0)$.
\end{cor}

\begin{proof}
By Fact \ref{fact317} and Corollary \ref{pro.embed}.
\end{proof}
Note that our definition of a PAC substructure implies that a PAC substructure is definably closed, which corresponds to being a perfect field in the case of the theory ACF, so one could wonder whether in the case of the theory ACF, projective profinite groups correspond to absolute Galois groups of PAC fields or \emph{perfect} PAC fields. In fact, they correspond to perfect PAC fields (see Corollary 23.1.2 in \cite{FrJa}).

\section{Projective profinite groups}
\subsection{PAC has projective absolute Galois group}
We start with a simple remark which helps in a better understanding property of being a PAC substructure.

\begin{cor}\label{PAC.def2}
A small substructure $P\subseteq\mathfrak{C}$ is PAC if and only if every stationary type over $P$ is finitely satisfable in $P$.
\end{cor}

The above corollary might be used as an alternative definition of being a PAC substructure. For more details about other possible versions of the definition of a PAC substructure the reader may consult subsection 3.1 in \cite{Hoff3}.

We note here an easy fact, which can be understood that ``sometimes" (see Lemma \ref{alg.ext.PAC}) being a PAC substructure might be uderstood as being ``one step before being a model" (it is enough to take the algebraic closure - if algebraic closure preserves PAC).

\begin{fact}[Corollary 3.10 in \cite{PilPol}]
Assume that $P$ is a small PAC substructure of $\mathfrak{C}$ such that $\acl_{\mathcal{L}}^{\mathfrak{C}}(P)=P$. Then $P\preceq \mathfrak{C}$.
\end{fact}

\begin{proof}
To see this, note that, since $\acl_{\mathcal{L}}^{\mathfrak{C}}(P)=P$, every extension of $P$ is regular. Therefore Tarski-Vaught test implies, that $P\preceq M$ for some small $M\preceq\mathfrak{C}$.
\end{proof}

The following example arose during discussions between Alex Kruckamn and Nick Ramsey, and Ludomir Newelski and us. The example shows that elimination of imaginaries is an important assumption for our purposes. To avoid such inconveniences, one might modify the definition of regularity as was proposed in Remark 3.2.(2) in \cite{Hoff3}. In this example we do not assume elimination of imaginaries for $T$.

\begin{example}\label{example.HKNR}
Consider a language $\mathcal{L}$ consisting only one relation $R$ and a theory $T$ stating that $R$ 
is an equivalence relation and that $R$
has only classes of size $3$. 
Note that $T$ is $\omega$-stable and has quantifier elimination.
Let us choose some monster model $\mathfrak{C}\models T$. We want to construct a PAC substructure of $\mathfrak{C}$.

To do this, consider $2$ countable and disjoint families of equivalence classes of $R$, say $A$ and $B$. Substructure $P$ consists of
\begin{itemize}
\item all elements from every equivalence class belonging to $A$,
\item exactly one element from every equivalence class belonging to $B$.
\end{itemize}
If $P\subseteq N$ is regular, then an intersection of $N$ with any equivalence class belonging to $B$ contains only one element, which we already chose for $P$. It follows that $P$ is existentially closed in $N$ and therefore $P$ is PAC.

Note that adding to $P$ only one element from every equivalence class belonging to $B$ will not produce a PAC substructure. To see this
let $N:=P\cup\lbrace a,b,c\rbrace$, where $\lbrace a,b,c\rbrace\cap P=\emptyset$, then $P\subseteq N$ is regular, but $N$ satisfies sentence stating that there exist three different elements $x$, $y$ and $z$, such that $R(x,y)$ and $R(y,z)$.

The absolute Galois group of $P$ is isomorphic to $(\mathbb{Z}/2\mathbb{Z})^{\omega}$, i.e.
$$\mathcal{G}(P)\cong (\mathbb{Z}/2\mathbb{Z})^{\omega}.$$
By Corollary 22.7.11 in \cite{FrJa}, the absolute Galois group of $P$ can not be projective.
\end{example}

Now, we generalize Lemma 1.17 from \cite{manuscript} (which states that the absolute Galois group of a PAC substructure - in the strongly minimal context - is projective) to our, i.e. stable, context. Of course, we still assume that $T$ is stable and has quantifier elimination and elimination of imaginaries.

\begin{theorem}\label{PAC.proj}
If a small $N$ is PAC, then $\mathcal{G}(N)$ is projective.
\end{theorem}

\begin{proof}
Assume that for some finite groups $A$ and $B$ we have continuous epimorphisms
$\rho:\mathcal{G}(N)\to A$ and
$\alpha:B\to A$. We will find 
$$\gamma:\mathcal{G}(N)\to B$$
such that $\rho=\alpha\gamma$,
$$\xymatrixcolsep{5pc}\xymatrix{\mathcal{G}(N)\ar@{->>}[r]^-{\rho} \ar@{-->}[dr]_-{\gamma}& A \\ & B \ar@{->>}[u]_{\alpha}}$$
\
\\
\\
\underline{$A\rightsquigarrow\aut_{\mathcal{L}}(L/N)$}

Because $\rho$ is continuous, $\ker\rho$ is a closed subgroup of $\mathcal{G}(N)$. The Galois correspondence, Fact \ref{galois.correspondence}, implies that for $L:=\acl_{\mathcal{L}}^{\mathfrak{C}}(N)^{\ker\rho}$, $L$ is definably closed and
$$\ker\rho=\aut_{\mathcal{L}}\Big(\acl_{\mathcal{L}}^{\mathfrak{C}}(N)/L\Big).$$
By Fact \ref{fact314}, $N\subseteq L$ is Galois. Thus, by Fact \ref{fact.exact}, $A\cong\aut_{\mathcal{L}}(L/N)$. Therefore, without loss of generality, we can assume that $A=\aut_{\mathcal{L}}(L/N)$ and $\rho=|_L$,
$$\xymatrixcolsep{5pc}\xymatrix{\mathcal{G}(N)\ar@{->>}[r]^-{|_L} & \aut_{\mathcal{L}}(L/N) \\ & B \ar@{->>}[u]_{\alpha}}$$
\
\\
\\
\underline{$B\rightsquigarrow\aut_{\mathcal{L}}(M/M^B)$}

By Fact \ref{cor.stationary_types_exist}, there exists a non-algebraic stationary type over $N$ (in the sense of $\mathfrak{C}$), say $p(x)$. Take elements from a Morley sequence of $p(x)$, 
say $b_1,\ldots, b_n\in\mathfrak{C}$, where $n:=|B|$,
such that
$$L\ind^{\mathfrak{C}}_N b_1\ldots b_n.$$
By multiple use of Fact \ref{lang413}, we obtain a sequence of regular extensions
$$N\subseteq Nb_1\subseteq Nb_1b_2\subseteq\ldots\subseteq Nb_1\ldots b_n,$$
hence, by Fact \ref{regular.remark}.(4), $N\subseteq Nb_1\ldots b_n$ is also regular. Stability implies that $b_1\ldots b_n$ is $N$-indiscernible as a set, hence for every $\sigma\in S_n$ (permutation of $n$-many elements) there exists $h_\sigma\in\aut_{\mathcal{L}}(\mathfrak{C}/N)$ such that $h_{\sigma}(b_i)=b_{\sigma(i)}$. Without loss of generality, $B\leqslant S_n$. Note that $B$ acts on $L$:
$$g(m)=\alpha(g)\big(m\big),$$
where $g\in B$ and $m\in L$. Moreover, $B$ acts on $b_1\ldots b_n$ by $h_\sigma$, where $\sigma\in B$.
Fact \ref{regular.PAPA} allows us to extend these both actions of group $B$ to an action on $M:=\dcl_{\mathcal{L}}^{\mathfrak{C}}(Lb_1\ldots b_n)$. Now, we will treat $B$ as a subgroup of $\aut_{\mathcal{L}}(M/M^B)$.

Since $B$ is finite, Lemma \ref{N_Galois}.(3) implies that $M^B\subseteq M$ is Galois and $M\subseteq\acl_{\mathcal{L}}^{\mathfrak{C}}(M^B)$. Moreover, $B$ as a finite group is a closed subgroup of $\aut_{\mathcal{L}}(M/M^B)$. Therefore, the Galois correspondence (Fact \ref{galois.correspondence}) and 
$$M^B=M^{\aut_{\mathcal{L}}(M/M^B)},$$
imply that $B=\aut_{\mathcal{L}}(M/M^B)$. Again, without loss of generality, we change the set-up:
$$\xymatrixcolsep{5pc}\xymatrix{\mathcal{G}(N)\ar@{->>}[r]^-{|_L} & \aut_{\mathcal{L}}(L/N) \\ & \aut_{\mathcal{L}}(M/M^B) \ar@{->>}[u]_{|_L}}$$
\
\\
\\
\underline{$N\subseteq M^B$ is regular}

(Similar argument is used in the proof of Proposition \ref{fr.ja.133})
Now, we will show that $M^B\cap L= N$. Of course $N\subseteq M^B\cap L$. Let $m\in M^B\cap L$, i.e. $m\in L$ and for all $\sigma\in B$ we have $m=\sigma(m)=\alpha(\sigma)(m)$. Because $\alpha$ is an epimorphism, it follows that for each $f\in\aut_{\mathcal{L}}(L/N)$ we have $f(m)=m$, hence $m\in N$.

Recall that $L\ind_N^{\mathfrak{C}}b_1\ldots b_n$, $N\subseteq Nb_1\ldots b_n$ is regular and $M=\dcl_{\mathcal{L}}^{\mathfrak{C}}(Lb_1\ldots b_n)$. Fact \ref{lang413} implies that $L\subseteq M$ is regular.
We have
$$M\cap\acl_{\mathcal{L}}^{\mathfrak{C}}(L)=L,\quad \acl_{\mathcal{L}}^{\mathfrak{C}}(L)=\acl_{\mathcal{L}}^{\mathfrak{C}}(N),$$
$$M\cap\acl_{\mathcal{L}}^{\mathfrak{C}}(N)=L,$$
$$M^B\cap\acl_{\mathcal{L}}^{\mathfrak{C}}(N)=M^B\cap L=N,$$
i.e. $N\subseteq M^B$ is regular. 
\
\\
\\
\underline{$N\rightsquigarrow N'$ which contains a copy of $M^B$}

Let $\bar{c}\subseteq\mathfrak{C}$ be such that $M^B=\dcl_{\mathcal{L}}^{\mathfrak{C}}(N\bar{c})$ ($\bar{c}$ can be an enumeration of $M^B$). Moreover, we introduce a set of $\mathcal{L}\cup\lbrace N\rbrace$-formulas,
$$q(\bar{x}):=\qftp_{\mathcal{L}}^{\mathfrak{C}}(\bar{c}/N)\;\cup\;\lbrace \bar{x}_0\subseteq N\;|\;\bar{x}_0\subseteq\bar{x}\rbrace.$$
Take a small $D\preceq\mathfrak{C}$ such that $N\subseteq D$.
\
\\
\textbf{Claim:} It follows that $\theo_{\mathcal{L}\cup\lbrace N\rbrace}(D)\cup q(\bar{x})$ is consistent.
\\
\textit{Proof of the claim:} Since $N$ is PAC and $N\subseteq M^B$ is regular, we have $N\preceq_1 M^B$. Therefore if $\varphi(n,\bar{x}_0)\in\qftp_{\mathcal{L}}^{\mathfrak{C}}(\bar{c}/N)$, it is $M^B\models\exists \bar{x}_0\;\varphi(n,\bar{x}_0)$ and so also $N\models\exists \bar{x}_0\;\varphi(n,\bar{x}_0)$ and $(D,N)\models(\exists \bar{x}_0)\big(\bar{x}_0\subseteq N\;\wedge\;\varphi(n,\bar{x}_0)\big)$.

Consider $(D',N')\succeq (D,N)$ which is $|N|^{+}$-saturated. Without loss of generality: $D\preceq D'\preceq\mathfrak{C}$. Note that $N'\succeq N$ and $\dcl_{\mathcal{L}}^{\mathfrak{C}}(N')=N'$.

There exists $\bar{c}'\subseteq N'$ such that $\bar{c}'\models\qftp_{\mathcal{L}}^{\mathfrak{C}}(\bar{c}/N)$. Quantifier elimination in $T$ implies that $\bar{c}'\models\tp_{\mathcal{L}}^{\mathfrak{C}}(\bar{c}/N)$, thus there exists $f\in\aut_{\mathcal{L}}(\mathfrak{C}/N)$ such that $f(\bar{c})=\bar{c}'$. Since $N\subseteq N\bar{c}$ is regular, Fact \ref{regular} allows us to assume that $f\in\aut_{\mathcal{L}}(\mathfrak{C}/\acl_{\mathcal{L}}^{\mathfrak{C}}(N))$.

Note that
$$f(M^B)=f\big(\dcl_{\mathcal{L}}^{\mathfrak{C}}(N\bar{c})\big)=\dcl_{\mathcal{L}}^{\mathfrak{C}}\big(Nf(\bar{c})\big)=\dcl_{\mathcal{L}}^{\mathfrak{C}}(N\bar{c}')\subseteq\dcl_{\mathcal{L}}^{\mathfrak{C}}(N')=N'$$
and
$$f(M)\subseteq f\big(\acl_{\mathcal{L}}^{\mathfrak{C}}(M^B)\big)=\acl_{\mathcal{L}}^{\mathfrak{C}}\big(f(M^B)\big)\subseteq \acl_{\mathcal{L}}^{\mathfrak{C}}(N').$$
We have a group isomorphism
$$F:\aut_{\mathcal{L}}(M/M^B)\ni h\mapsto fhf^{-1}\in\aut_{\mathcal{L}}\big(f(M)/f(M^B)\big).$$

Because $N\preceq N'$, we conclude, by Fact \ref{lang410}, that $N\subseteq N'$ is regular. Hence, by Fact \ref{regular}, the following map
$$H:\aut_{\mathcal{L}}\Big(\dcl_{\mathcal{L}}^{\mathfrak{C}}\big(\acl_{\mathcal{L}}^{\mathfrak{C}}(N),N'\big)/ N'\Big)\ni h\mapsto h|_{\acl_{\mathcal{L}}^{\mathfrak{C}}(N)}\in\mathcal{G}(N)$$
is onto and therefore a group isomorphism.
\
\\
\\
\underline{Almost final diagram}

Since $A\cong\aut_{\mathcal{L}}(L/N)$ is finite, we can choose a finite $\bar{a}\subseteq\acl_{\mathcal{L}}^{\mathfrak{C}}(N)$. $|\bar{a}|=m$, such that $L=\dcl_{\mathcal{L}}^{\mathfrak{C}}(N\bar{a})$ and $\aut_{\mathcal{L}}(L/N)\cdot\bar{a}=\bar{a}$. Thus $M=\dcl_{\mathcal{L}}^{\mathfrak{C}}(L\bar{b})=\dcl_{\mathcal{L}}^{\mathfrak{C}}(N\bar{a}\bar{b})$.
To this point, we have:
$$\xymatrixcolsep{1.1pc}\xymatrix{\mathcal{G}(N)\ar[d]_{H^{-1}}^{\cong}\ar@{->>}[rr]^-{|_L} & & \aut_{\mathcal{L}}(L/N) \\ 
\aut_{\mathcal{L}}\Big(\faktor{\dcl_{\mathcal{L}}^{\mathfrak{C}}\big(\acl_{\mathcal{L}}^{\mathfrak{C}}(N),N'\big)}{ N'}\Big) \ar@{-->}[dd] & & \aut_{\mathcal{L}}(M/M^B) \ar@{->>}[u]_{|_L} \\
 & \aut_{\mathcal{L}}\Big(\faktor{\dcl_{\mathcal{L}}^{\mathfrak{C}}\big(L,f(\bar{b})\big)}{f(M^B)} \Big)\ar@{-}[r]^-{=}& \aut_{\mathcal{L}}\big(f(M)/f(M^B)\big) \ar[u]_{F^{-1}}^{\cong} \\
 ? \ar@{-->}[r] & \aut_{\mathcal{L}}\Big(\faktor{\dcl_{\mathcal{L}}^{\mathfrak{C}}\big(N',\bar{a},f(\bar{b})\big)}{N'} \Big) \ar@{-}[r]_-{=} &  \aut_{\mathcal{L}}\Big(\faktor{\dcl_{\mathcal{L}}^{\mathfrak{C}}\big(N',L,f(\bar{b})\big)}{N'} \Big) \ar[ul]_(.35){|_{\dcl_{\mathcal{L}}^{\mathfrak{C}}(L,f(\bar{b}))}}
 }$$
We are done if we can find a proper group in the place of ``?".

\
\\
\\
\underline{Finding $\bar{b'}$}

Since $f(M^B)\subseteq N'$, $f(M^B)\subseteq f(M)$ is Galois and $\aut_{\mathcal{L}}(M/M^B)\cdot\bar{b}=B\cdot\bar{b}=\bar{b}$,

\begin{IEEEeqnarray*}{rCl}
\aut_{\mathcal{L}}(\mathfrak{C}/N')\cdot f(\bar{b}) &\subseteq & 
\aut_{\mathcal{L}}(\mathfrak{C}/f(M^B))\cdot f(\bar{b}) \\
&=& \aut_{\mathcal{L}}(f(M)/f(M^B))\cdot f(\bar{b})\\
&=&f\big(\aut_{\mathcal{L}}(M/M^B)\cdot \bar{b}\big)\\
&=&f(\bar{b}).
\end{IEEEeqnarray*}
Let
\begin{IEEEeqnarray*}{rCl}
\aut_{\mathcal{L}}(\mathfrak{C}/N')\cdot f(\bar{b}) & = & f(\bar{b})=\\
& = & \aut_{\mathcal{L}}(\mathfrak{C}/N')\cdot f(b_{i_1})\;\cupdot\; \ldots \;
\cupdot\; \aut_{\mathcal{L}}(\mathfrak{C}/N')\cdot f(b_{i_s})
\end{IEEEeqnarray*}
For each $k\leqslant s$ we choose $\varphi(d,y)$ ($d\subseteq N'$ will be dynamically extended...) such that
$$\aut_{\mathcal{L}}(\mathfrak{C}/N')\cdot f(b_{i_k})=\varphi_k(d,\mathfrak{C}).$$
Note that $\aut_{\mathcal{L}}\Big(\dcl_{\mathcal{L}}^{\mathfrak{C}}\big(N',\bar{a},f(\bar{b})\big)/N' \Big)$, as determined by values on $\bar{a}f(\bar{b})$, is finite and if there is no $h\in\aut_{\mathcal{L}}\Big(\dcl_{\mathcal{L}}^{\mathfrak{C}}\big(N',\bar{a},f(\bar{b})\big)/N' \Big)$ such that
$$h\big(a_1\ldots a_mf(b_1)\ldots f(b_n)\big)=a_{\sigma(1)}\ldots a_{\sigma(m)} f(b_{\sigma'(1)})\ldots f(b_{\sigma'(n)}),$$
then there is no such $h$ in $\aut_{\mathcal{L}}(\mathfrak{C}/N')$ (since $N'\subseteq \dcl_{\mathcal{L}}^{\mathfrak{C}}\big(N',\bar{a},f(\bar{b})\big)$ is normal) and hence
\begin{equation}\tag{$\ast$}
a_1\ldots a_mf(b_1)\ldots f(b_n)\not\equiv_{N'}a_{\sigma(1)}\ldots a_{\sigma(m)} f(b_{\sigma'(1)})\ldots f(b_{\sigma'(n)}).
\end{equation}
We choose a formula $\psi_{\sigma,\sigma'}$ such that
$$\models \psi_{\sigma,\sigma'}\big(d,a_1,\ldots,a_m,f(b_1),\ldots,f(b_n)\big),$$
$$\models \neg\psi_{\sigma,\sigma'}\big(d,a_{\sigma(1)},\ldots,a_{\sigma(m)},f(b_{\sigma'(1)}),\ldots,f(b_{\sigma'(n)})\big).$$
Now, we introduce an $\mathcal{L}$-formula $\theta(d,\bar{a})$ given by
\begin{IEEEeqnarray*}{rl}
(\exists y_1,\ldots, y_n) &\Big(\;\bigwedge\limits_{k<i_2}\varphi_{i_1}(d,y_k)\;\wedge\;(\forall y)\big(\varphi_{i_1}(d,y)\rightarrow \bigvee\limits_{k<j_2}y=y_k\big)\;\wedge\\
& \qquad\qquad\qquad\vdots\\
& \quad\bigwedge\limits_{i_s\leqslant k\leqslant n}\varphi_{i_s}(d,y_k)\;\wedge\;(\forall y)\big(\varphi_{i_s}(d,y)\rightarrow \bigvee\limits_{i_s\leqslant k\leqslant n}y=y_k\big)\;\wedge\\
& \bigwedge\limits_{(\sigma,\sigma')\text{ as in }(\ast)}
\psi_{\sigma,\sigma'}\big(d,a_1,\ldots,a_m,y_1,\ldots,y_n\big)\;\wedge \\
& \qquad\neg\psi_{\sigma,\sigma'}\big(d,a_{\sigma(1)},\ldots,a_{\sigma(m)},y_{\sigma'(1)},\ldots,y_{\sigma'(n)}\big)
\Big).
\end{IEEEeqnarray*}
We have
$$(D',N')\models \theta(d,\bar{a}),$$
$$(D',N')\models \exists z\;\big(z\subseteq N'\;\wedge\;\theta(z,\bar{a})\big),$$
$$(D,N)\models \exists z\;\big(z\subseteq N\;\wedge\;\theta(z,\bar{a})\big).$$
Let $d'\subseteq N$ be such that $D\models \theta(d'\bar{a})$ and let $b_1',\ldots,b_n'\subseteq D$ witness existence of $y_1,\ldots,y_n$ for $\theta(d',\bar{a})$ in $D$. We see that $\bar{b}'\subseteq\acl_{\mathcal{L}}^{\mathfrak{C}}(N)$ and
$$\aut_{\mathcal{L}}(\mathfrak{C}/N')\cdot\bar{b}'\subseteq\aut_{\mathcal{L}}(\mathfrak{C}/N)\cdot\bar{b}'=\bar{b}'.$$
Therefore $N'\subseteq\dcl_{\mathcal{L}}^{\mathfrak{C}}(N',\bar{a},\bar{b}')$ is Galois and there is a restriction map:
$$\aut_{\mathcal{L}}\Big(\dcl_{\mathcal{L}}^{\mathfrak{C}}\big(\acl_{\mathcal{L}}^{\mathfrak{C}}(N), N'\big)/N'\Big)\to
\aut_{\mathcal{L}}\big(\dcl_{\mathcal{L}}^{\mathfrak{C}}(N',\bar{a},\bar{b}')/N'\big).$$
\
\\
\\
\underline{Final diagram}

The last map we need is the following one
$$\Delta:\aut_{\mathcal{L}}\big(\dcl_{\mathcal{L}}^{\mathfrak{C}}(N',\bar{a},\bar{b}')/N'\big)\to\aut_{\mathcal{L}}\big(\dcl_{\mathcal{L}}^{\mathfrak{C}}(N',\bar{a},f(\bar{b}))/N'\big)$$
and we define it in the following way. Let $h\in\aut_{\mathcal{L}}\big(\dcl_{\mathcal{L}}^{\mathfrak{C}}(N',\bar{a},\bar{b}')/N'\big)$ and
$$h(a_1\ldots a_m b_1'\ldots b_n')=a_{\sigma(1)}\ldots a_{\sigma(m)} b_{\sigma'(1)}'\ldots b_{\sigma(n)}'.$$
There exists $\tilde{h}\in \aut_{\mathcal{L}}\big(\dcl_{\mathcal{L}}^{\mathfrak{C}}(N',\bar{a},f(\bar{b}))/N'\big)$ such that
$$\tilde{h}\big(a_1\ldots a_m f(b_1)\ldots f(b_n)\big)=a_{\sigma(1)}\ldots a_{\sigma(m)} f(b_{\sigma'(1)})\ldots f(b_{\sigma(n)})$$
(otherwise $(\sigma,\sigma')$ would satisfy $(\ast)$, but then
$$\models \psi_{\sigma,\sigma'}\big(d',a_1,\ldots,a_m,b_1',\ldots,b_n'\big),$$
$$\models \neg\psi_{\sigma,\sigma'}\big(d',a_{\sigma(1)},\ldots,a_{\sigma(m)},b_{\sigma'(1)}',\ldots,b_{\sigma'(n)}'\big)$$
which contradicts the existence of $h$). We set $\Delta(h):=\tilde{h}$.

Now, we put all the above together:
$$\xymatrixcolsep{1.1pc}\xymatrix{\mathcal{G}(N)\ar[d]_{H^{-1}}^{\cong}\ar@{->>}[rr]^-{|_L} & & \aut_{\mathcal{L}}(L/N) \\ 
\aut_{\mathcal{L}}\Big(\faktor{\dcl_{\mathcal{L}}^{\mathfrak{C}}\big(\acl_{\mathcal{L}}^{\mathfrak{C}}(N),N'\big)}{ N'}\Big) \ar[dd]_{|_{\dcl_{\mathcal{L}}^{\mathfrak{C}}(N',\bar{a},\bar{b}')}} & & \aut_{\mathcal{L}}(M/M^B) \ar@{->>}[u]_{|_L} \\
 & \aut_{\mathcal{L}}\Big(\faktor{\dcl_{\mathcal{L}}^{\mathfrak{C}}\big(L,f(\bar{b})\big)}{f(M^B)} \Big)\ar@{-}[r]^-{=}& \aut_{\mathcal{L}}\big(f(M)/f(M^B)\big) \ar[u]_{F^{-1}}^{\cong} \\
 \aut_{\mathcal{L}}\Big(\faktor{\dcl_{\mathcal{L}}^{\mathfrak{C}}(N',\bar{a},\bar{b}')}{N'}\Big) \ar[r]_-{\Delta} & \aut_{\mathcal{L}}\Big(\faktor{\dcl_{\mathcal{L}}^{\mathfrak{C}}\big(N',\bar{a},f(\bar{b})\big)}{N'} \Big) \ar@{-}[r]_-{=} &  \aut_{\mathcal{L}}\Big(\faktor{\dcl_{\mathcal{L}}^{\mathfrak{C}}\big(N',L,f(\bar{b})\big)}{N'} \Big) \ar[ul]_(.35){|_{\dcl_{\mathcal{L}}^{\mathfrak{C}}(L,f(\bar{b}))}}
 }$$
 
The above diagram commutes, since ``the long path" does not do anything with values of autmorphisms on $L=\dcl_{\mathcal{L}}^{\mathfrak{C}}(N\bar{a})$.
\end{proof}

\subsection{Projective profinite group as absolute Galois group}
In this subsection, we show that property PP$_T$ holds for a subclass of the class of stable theories. The only issue not allowing us to extend our result over all stable theories is the fact that sometimes 
algebraic closure of a PAC substructure is not PAC, which seems rather strange if we remember that ``PAC" states for ``pseudo-algebraically closed". However, many interesting stable theories satisfy a simplified version of the main assumption of Lemma \ref{alg.ext.PAC}:
\begin{center}
a type over $A$ has only finitely many extensions over $\acl_{\mathcal{L}}^{\mathfrak{C}}(A)$,
\end{center}
which holds for any type in e.g. any $\omega$-stable theory.
Because we did not achieve property PP$_T$ for arbitrary stable $T$,
we consider a modification of property PP$_T$:
\begin{enumerate}
\item[(PP$_T^{\ast}$)] if profinite $G$ is projective then $G\cong\mathcal{G}(P)$ for some definably closed substructure $P$ of $\mathfrak{C}$
\end{enumerate}
which holds for any stable theory $T$ (with quantifier elimination and elimination of imaginaries) - see Theorem \ref{fr.ja.2311}.
Since the absolute Galois group of a PAC substructure is projective for any stable $T$, the right-to-left implication in PP$_T$ is true for any stable $T$. Therefore property PP$_T^{\ast}$ might be understood as a weaker version of the left-to-right implication in PP$_T$.

The following lemma is a simple modification of Proposition 3.9 in \cite{PilPol} (related to our alternative definition of a PAC substructure), which generalizes a well known fact about PAC fields: any algebraic extension of a PAC field is PAC field. Proof of our slight modification is based on the original proof of Proposition 3.9 in \cite{PilPol}, but for the reader's convenience, instead of listing all the small differences, we provide the whole proof.

\begin{lemma}\label{alg.ext.PAC}
Let $P$ be a small PAC substructure of $\mathfrak{C}$, and let 
$P\subseteq Q=\dcl_{\mathcal{L}}^{\mathfrak{C}}(Q)\subseteq Q'\subseteq\acl_{\mathcal{L}}^{\mathfrak{C}}(P)$, where $Q$ and $Q'$ are substructures of $\mathfrak{C}$ such that $P\subseteq Q'$ is normal (e.g. $Q'=\acl_{\mathcal{L}}^{\mathfrak{C}}(P)$).
If any type over $P$ has only finitely many (non-forking) extensions over $Q'$, then $Q$ is PAC.
\end{lemma}

\begin{proof}
We want to show that if a type $p(x)$ over $Q$ is stationary, then it is finitely satisfable in $Q$ (as in Corollary \ref{PAC.def2}). 
Assume that $\varphi(m_0,x)\in p(x)$.

There are only finitely many distinct extensions of $p|_P$ over $Q'$, say $p_1,\ldots,p_n\in S(Q')$. We assume that $p_1\supseteq p$, so $p_1$ is stationary. Since $P\subseteq Q'$ is normal, for every $i\leqslant n$ the type $p_i$ is stationary. Therefore there are only finitely many distinct extensions of type $p|_P$ over $\acl_{\mathcal{L}}^{\mathfrak{C}}(P)$, abusing notation: $p_1,\ldots,p_n\in S(\acl_{\mathcal{L}}^{\mathfrak{C}}(P))$.

Consider 
$$\tilde{p}:=\bigotimes\limits_{i\leqslant n}p_i \in S(\acl_{\mathcal{L}}^{\mathfrak{C}}(P)),$$
some $d_1\ldots d_n\models\tilde{p}$ and the code $d'$ for the set $\lbrace d_1,\ldots,d_n\rbrace$ (\textit{here, we are using that there are only finitely many extensions}).
\
\\
\textbf{Claim:} It follows that $\tp_{\mathcal{L}}^{\mathfrak{C}}(d'/P)$ is stationary.
\\
\textit{Proof of the claim:} It is enough to show that there is only one extension of the type
$\tp_{\mathcal{L}}^{\mathfrak{C}}(d'/P)$ over $\acl_{\mathcal{L}}^{\mathfrak{C}}(P)$. Let $\phi(c,y)\in
\tp_{\mathcal{L}}^{\mathfrak{C}}(d'/\acl_{\mathcal{L}}^{\mathfrak{C}}(P))$, i.e.
$$\mathfrak{C}\models\phi(c,d'),$$
and let $f\in\aut_{\mathcal{L}}(\mathfrak{C}/P)$. Since $\lbrace d_1,\ldots,d_n\rbrace$ is $P$-independent, it follows that $\lbrace f(d_1),\ldots,f(d_n)\rbrace$ is also $P$-independent and so $\acl_{\mathcal{L}}^{\mathfrak{C}}(P)$-independent. Note that there exists some permutation $\sigma\in S_n$ such that $f(d_{\sigma(i)})\models p_i$, hence 
$$f(d_{\sigma(1)})\ldots f(d_{\sigma(n)})\models\tilde{p}.$$
There exists $h\in\aut_{\mathcal{L}}(\mathfrak{C}/\acl_{\mathcal{L}}^{\mathfrak{C}}(P))$ such that
$$f(d_{\sigma(1)})\ldots f(d_{\sigma(n)})=h(d_1)\ldots h(d_n).$$
Therefore for each $i\leqslant n$ it follows that $h^{-1}f(d_{\sigma(i)})=d_i$, so $h^{-1}f(d')=d'$.
We have
$$\mathfrak{C}\models \phi(h^{-1}f(c),h^{-1}f(d')),$$
$$\mathfrak{C}\models \phi(h^{-1}f(c),d'),$$
but, since $c\in\acl_{\mathcal{L}}^{\mathfrak{C}}(P)$, we have $h^{-1}f(c)=f(c)$ and the previous line can be written as
$$\mathfrak{C}\models \phi(f(c),d').$$
Because $f\in\aut_{\mathcal{L}}(\mathfrak{C}/P)$ was arbitrary, we have shown the claim.

Note that $d_1\in\dcl_{\mathcal{L}}^{\mathfrak{C}}(Q,d')$. To see this, take any $f\in\aut_{\mathcal{L}}(\mathfrak{C}/Qd')$. Since $f(d')=d'$, it follows that $f(d_1)=d_i$ for some $i\leqslant n$. We have $\tp_{\mathcal{L}}^{\mathfrak{C}}(d_i/Q)=\tp_{\mathcal{L}}^{\mathfrak{C}}(d_1/Q)=p$, which is stationary. Therefore $p_i$ and $p_1$ are non-forking extensions of a stationary type and so $p_i=p_1$ and $d_i=d_1$.

There exists a $\mathcal{L}$-formula $\theta$ such that $\theta(q_0,d',\mathfrak{C})=\lbrace d_1\rbrace$. We have
$$(\exists x)\Big(\bigvee\limits_{f\in\aut_{\mathcal{L}}(\mathfrak{C}/P)} \;\theta\big(f(q_0),y,x\big)\;\wedge\;\exists !x'\,\theta\big(f(q_0),y,x'\big)\;\wedge\;\varphi(m_0,x)\Big)\in\tp_{\mathcal{L}}^{\mathfrak{C}}(d'/P).$$
Since $P$ is PAC, and $\tp_{\mathcal{L}}^{\mathfrak{C}}(d'/P)$ is stationary, there exists $a'\in P$, $b\in\mathfrak{C}$ and $f\in\aut_{\mathcal{L}}(\mathfrak{C}/P)$ such that
$$\mathfrak{C}\models \theta\big(f(q_0),a',b\big)\;\wedge\;\exists !x'\,\theta\big(f(q_0),a',x'\big)\;\wedge\;\varphi(m_0,b),$$
$$\mathfrak{C}\models \theta\big(q_0,a',f^{-1}(b)\big)\;\wedge\;\exists !x'\,\theta(q_0,a',x')\;\wedge\;\varphi\big(m_0,f^{-1}(b)\big).$$
It follows that $f^{-1}(b)\in \dcl_{\mathcal{L}}^{\mathfrak{C}}(q_0,a')\subseteq Q$ and $\mathfrak{C}\models \varphi\big(m_0,f^{-1}(b)\big)$, what ends the proof.
\end{proof}

\begin{remark}\label{counterexamples.alg.PAC}
One could ask about possible generalizations of the above lemma. 
Section 5. in \cite{PilPol} provides an example of a superstable theory $T$ and a bounded PAC substructure $P$ of $\mathfrak{C}$ such that $\acl_{\mathcal{L}}^{\mathfrak{C}}(P)$ is not an elementary substructure (recall that an algebraically closed PAC substructure is an elementary substructure).
Therefore, it looks that there is no natural generalization of Lemma \ref{alg.ext.PAC}.
\end{remark}

\begin{cor}\label{cor.acl.PAC}
If small $P\subseteq\mathfrak{C}$ is PAC and there are only finitely many extensions over $\acl_{\mathcal{L}}^{\mathfrak{C}}(P)$ of every type over $P$, then $\acl_{\mathcal{L}}^{\mathfrak{C}}(P)$ is PAC and $\acl_{\mathcal{L}}^{\mathfrak{C}}(P)\preceq\mathfrak{C}$.
\end{cor}

\begin{question}\label{que1}
Assume that $P$ is PAC.
\begin{enumerate}
\item What are the obstacles to show that $\acl_{\mathcal{L}}^{\mathfrak{C}}(P)$ is PAC?

\item Assume moreover that $\acl_{\mathcal{L}}^{\mathfrak{C}}(P)$ is PAC.
Does every type over $P$ have only finitely many extensions over $\acl_{\mathcal{L}}^{\mathfrak{C}}(P)$? If not, then investigate a counterexample.
\end{enumerate}
\end{question}

The following proposition generalizes Theorem 23.1.1. in \cite{FrJa}.

\begin{theorem}\label{fr.ja.2311}
Assume that $N_0\subseteq N$ is a Galois extension of small substructures of $\mathfrak{C}$ and assume that there is an epimorhism of profinite groups
$\alpha:G\to\aut_{\mathcal{L}}(N/N_0)$, and $G$ is projective.
there exists a definably closed substructure $P\supseteq N_0$ of $\mathfrak{C}$ 
and an isomorphism of profinite groups $\gamma:G\to\mathcal{G}(P)$ such that
$$\xymatrix{G \ar[r]^-{\gamma} \ar[dr]_-{\alpha} & \mathcal{G}(P) \ar[d]^{|_N} \\
& \aut_{\mathcal{L}}(N/N_0)
}$$
is commuting.
Moreover, if any type over $A$ has only finitely many extensions over $\acl_{\mathcal{L}}^{\mathfrak{C}}(A)$, then $P$ is PAC.
\end{theorem}

\begin{proof}
By Proposition \ref{fr.ja.133}, there exist regular extensions $N_0\subseteq M_0$ and $N\subseteq M$ such that $M_0\subseteq M$ is Galois and $\beta:G\cong\aut_{\mathcal{L}}(M/M_0)$ such that
$$\xymatrix{\aut_{\mathcal{L}}(M/M_0) \ar[dr]^-{|_N} & \\
G \ar[u]^-{\beta} \ar[r]_-{\alpha} & \aut_{\mathcal{L}}(N/N_0)
}$$
is commuting.
By Proposition 3.6 in \cite{Hoff3}, there is a PAC substructure $M_0'$ such that $M_0\subseteq M_0'$ is regular. Because $M_0\subseteq M_0'$ is regular and $\acl_{\mathcal{L}}^{\mathfrak{C}}(M_0)=\acl_{\mathcal{L}}^{\mathfrak{C}}(M)$, it follows that $M_0'\cap M=M_0$. 
We obtain that the restriction map
$$\aut_{\mathcal{L}}(\dcl_{\mathcal{L}}^{\mathfrak{C}}(M,M_0')/M_0')\xrightarrow{|_M}\aut_{\mathcal{L}}(M/M_0)$$
is an isomorphism. By $w:\aut_{\mathcal{L}}(M/M_0)\to\aut_{\mathcal{L}}(\dcl_{\mathcal{L}}^{\mathfrak{C}}(M,M_0')/M_0') $ we denote  the inverse of the restriction map $|_M$.
Hence $\aut_{\mathcal{L}}(\dcl_{\mathcal{L}}^{\mathfrak{C}}(M,M_0')/M_0')\cong G$ is projective, and so the restriction map
$$\mathcal{G}(M_0')\xrightarrow{|_{\dcl_{\mathfrak{C}}(M,M_0')}}\aut_{\mathcal{L}}(\dcl_{\mathcal{L}}^{\mathfrak{C}}(M,M_0')/M_0')$$
has a section $i:\aut_{\mathcal{L}}(\dcl_{\mathcal{L}}^{\mathfrak{C}}(M,M_0')/M_0')\to\mathcal{G}(M_0')$,
$$\xymatrixcolsep{4pc}\xymatrix{\mathcal{G}(M_0') \ar@/^2pc/[r]^{|_{\dcl_{\mathfrak{C}}(M,M_0')}}  & \aut_{\mathcal{L}}(\dcl_{\mathcal{L}}^{\mathfrak{C}}(M,M_0')/M_0') \ar[d]^{|_M} \ar[l]^{i} & \\
& \aut_{\mathcal{L}}(M/M_0) \ar[dr]^-{|_N} \ar@/^1pc/[u]^-{w}& \\
& G \ar[u]^-{\beta} \ar[r]_-{\alpha} & \aut_{\mathcal{L}}(N/N_0)
}$$
We set $\gamma:=i\circ w\circ\beta$ (note that $\gamma:G\to\mathcal{G}(M_0')$ is a continuous embedding),
$$P:=\acl_{\mathcal{L}}^{\mathfrak{C}}(M_0')^{\gamma(G)}$$
and note that $\mathcal{G}(P)=\gamma(G)$ and
$$\xymatrix{G \ar[r]^-{\gamma} \ar[dr]_-{\alpha} & \mathcal{G}(P) \ar[d]^-{|_N} \\
& \aut_{\mathcal{L}}(N/N_0)
}$$
is commuting.
The ``moreover" part follows from Corollary \ref{cor.acl.PAC}.
\end{proof}

\begin{cor}\label{pro.pro.embed}
Assume that $T$ is $\omega$-stable. Then property PP$_T$ holds, i.e.
a profinite group $G$ is projective if and only if $G$ is isomorphic to the absolute Galois group of some PAC substructure of $\mathfrak{C}$.
\end{cor}

\begin{question}\label{que2}
What is the biggest class of stable theories $T$ for which property PP$_T$ holds?
\end{question}

\begin{cor}\label{cor.counter.counter.example}
(For any stable theory $T$ with quantifier elimination and elimination of imaginaries.)
There exists a finite subset $A$ of $\mathfrak{C}$ such that
$$\acl_{\mathcal{L}}^{\mathfrak{C}}(A)\neq\dcl_{\mathcal{L}}^{\mathfrak{C}}\big(A,\;\acl_{\mathcal{L}}^{\mathfrak{C}}(\emptyset)\big).$$
\end{cor}

\begin{proof}
Reductio ad absurdum. Suppose that for every finite $A$ we have
$$\acl_{\mathcal{L}}^{\mathfrak{C}}(A)=\dcl_{\mathcal{L}}^{\mathfrak{C}}\big(A,\;\acl_{\mathcal{L}}^{\mathfrak{C}}(\emptyset)\big).$$
It follows that the above holds also for each small (finite or non-finite) set $A$. Hence for every small definably closed $A\subseteq\mathfrak{C}$, we have
\begin{IEEEeqnarray*}{rCl}
\mathcal{G}(A) &=& \aut_{\mathcal{L}}
\Big(\dcl_{\mathcal{L}}^{\mathfrak{C}}\big(A,\;\acl_{\mathcal{L}}^{\mathfrak{C}}(\emptyset)\big)/A\Big) \\
&\cong &
\aut_{\mathcal{L}}\big(\acl_{\mathcal{L}}^{\mathfrak{C}}(\emptyset)/A\cap\acl_{\mathcal{L}}^{\mathfrak{C}}(\emptyset) \big)\leqslant \mathcal{G}\big(\dcl_{\mathcal{L}}^{\mathfrak{C}}(\emptyset)\big).
\end{IEEEeqnarray*}
Note that $|\acl_{\mathcal{L}}^{\mathfrak{C}}(\emptyset)|$, and so $|\mathcal{G}(\dcl_{\mathcal{L}}^{\mathfrak{C}}(\emptyset))|$, depends only on $|T|$ and not on the saturation of the monster model (i.e. on $\kappa$).
Let $G$ be a profinite projective group of the size $2^{|\mathcal{G}(\dcl_{\mathcal{L}}^{\mathfrak{C}}(\emptyset))|}$ (such a group exists by existence of arbitrarily large profinite groups, Theorem 3.3.16 and Lemma 7.6.3 in \cite{ribzal}). By Theorem \ref{fr.ja.2311} there exists a small definably closed substructure $P$ such that
$$G\cong\mathcal{G}(P)\leqslant \mathcal{G}\big(\dcl_{\mathcal{L}}^{\mathfrak{C}}(\emptyset)\big).$$
\end{proof}

Note that equalities similar to the one from Corollary \ref{cor.counter.counter.example}, i.e. of the form
$$\acl_{\mathcal{L}}^{\mathfrak{C}}(B)=\dcl_{\mathcal{L}}^{\mathfrak{C}}(B,\acl_{\mathcal{L}}^{\mathfrak{C}}(A)),$$
where $A\subseteq B$, are desired in model theory (e.g. check the discussion before Conjecture \ref{conjecture1}, Definition 4.19 in \cite{Hoff3} , or Proposition 2.5 in \cite{PilPol}).

\section{G-actions on substructures}\label{sec.G.actions}
Assume that $G$ is \textbf{finitely generated}, $T$ is a stable $\mathcal{L}$-theory, which has quantifier elimination and elimination of imaginaries. We denote by ``$\hat{G}$" the profinite completion of group $G$.

Now, we are interested in (absolute Galois groups of) substructures of a monster model $\mathfrak{C}$ of the theory $T$, which are equipped with a group action of the group $G$. Note that an $\mathcal{L}$-structure $M$ might be seen as a substructure of $\mathfrak{C}$ if and only if $M\models\ T_{\forall}$, i.e. $T_{\forall}$ is the theory of the class of small $\mathcal{L}$-substructures of $\mathfrak{C}$. We introduce a new language
$$\mathcal{L}^G:=\mathcal{L}\;\cup\;\{\sigma_g\}_{g\in G},$$
where each $\sigma_g$ is a unary function symbol (but, for simplicity, it will denote also the interpretation of $\sigma_g$ in an $\mathcal{L}^G$-structure $(M,\bar{\sigma})$, ``$\sigma_g^M$"). A $\mathcal{L}^G$-structure $(M,\bar{\sigma})$ is a model of $(T_{\forall})_G$ if and only if
\begin{itemize}
\item $M\models T_{\forall}$,
\item for each $g\in G$, it follows that $\sigma_g\in\aut_{\mathcal{L}}(M)$,
\item $G\ni g\mapsto \sigma_g\in\aut_{\mathcal{L}}(M)$ is a homomorphism of groups.
\end{itemize}
Assume that $(M,\bar{\sigma})$ is an existentially closed model of $(T_{\forall})_G$ (i.e. existentially closed among all small substructures of $\mathfrak{C}$ equipped with an action of group $G$).

We are interested in a description of $\mathcal{G}(M)$. It turned out that it is good to start with the  description of $\mathcal{G}(M^G)$ (it is easier and might be used in the desired description of $\mathcal{G}(M)$). The idea behind next results is the following one: action of group $G$ on $M$ depends only on the action of group $G$ on $M\cap\acl_{\mathcal{L}}^{\mathfrak{C}}(M^G)$ - the relative algebraic closure of invariants.
The following proposition, which is a generalization of Theorem 4. in \cite{sjogren}, partially express this idea   by embedding the group $G$ into the group of automorphisms of the relative algebraic closure of invariants.

\begin{prop}\label{profinite.completion}
It follows that
$$\mathcal{A}:=\aut_{\mathcal{L}}(M\cap\acl_{\mathcal{L}}^{\mathfrak{C}}(M^G)/M^G)\cong \hat{G}.$$
\end{prop}

\begin{proof}
By Proposition 3.31 in \cite{Hoff3}, $\mathcal{A}$ is finitely generated. Corollary 3.2.8 in \cite{ribzal} says that two finitely generated groups have isomorphic profinite completions if and only if they have the same finite quotients:
$$\hat{\mathcal{A}}\cong\hat{G}\qquad\text{iff}\qquad\imG(\mathcal{A})=\imG(G).$$
We need to evoke a significant theorem: Theorem 1.1 in \cite{NikoSega} - profinite completion of a finitely generated profinite group is equal to this group (equivalently: its subgroups of finite index are open). Hence $\mathcal{A}=\hat{\mathcal{A}}$. Moreover, every homomorphism $\alpha:\mathcal{A}\to H$, where $H$ is a finite group with discrete topology, is continuous and therefore $\imG(\mathcal{A})\subseteq\imG(G)$.
We need to show that $\imG(\mathcal{A})\supseteq\imG(G)$.

Let $\pi: G\twoheadrightarrow H$ be a homomorphism onto a finite group $H$ and let $m:=|H|$. Our goal is to prove existence of a surjective group homomorphism $\mathcal{A}\twoheadrightarrow H$.

Take $\bar{b}=(b_1,\ldots,b_m)$, where $b_i$ are pairwise different elements of a Morley sequence in some stationary type $p(x)$ over $M^G$ (which exists by Fact \ref{cor.stationary_types_exist}). Without loss of generality, we assume that $M\ind_{M^G}^{\mathfrak{C}}\bar{b}$. We treat $H$ as a subgroup of $S_m$ and since the $\lbrace b_1,\ldots,b_m\rbrace$ is $M^G$-indiscernible, $H$ acts on $N:=\dcl_{\mathcal{L}}^{\mathfrak{C}}(M^G\bar{b})$. Moreover, $G$ acts on $N$ by $g\cdot a=\pi(g)(a)$ for $a\in N$. By Fact \ref{regular.PAPA} the action of $G$ on $M$ and the action of $G$ on $N$ extend simultaneously to an action of group $G$ on $M':=\dcl_{\mathcal{L}}^{\mathfrak{C}}(M\bar{b})$ (similarly as in the proof of Theorem \ref{PAC.proj}).
Note that, by Lemma \ref{N_Galois}, $N^H\subseteq N$ is Galois (in particular $N\subseteq\acl_{\mathcal{L}}^{\mathfrak{C}}(N^H)$). Since $\pi$ is onto, it follows that $M^G\subseteq N^H\subseteq (M')^G$, thus $(M,\bar{\sigma})\subseteq (M',\bar{\sigma}')$, where $\bar{\sigma}'$ is the above defined action of $G$ on $M'$. Existentially closedness of $(M,\bar{\sigma})$ implies that $(M,\bar{\sigma})\preceq_1(M,\bar{\sigma}')$.

Let us fix some $|M|^{+}$-strongly homogeneous $D\preceq\mathfrak{C}$ such that $M\subseteq D$ and $D\ind_M^{\mathfrak{C}}M'$.
For sure, every $\sigma_g$ extends to an element of $\aut_{\mathcal{L}}(D)$, and for simplicity we denote such an extension by the same symbol ``$\sigma_g$". Since $D\ind_M^{\mathfrak{C}}M'$ and $M\subseteq M'$ is regular (by Fact \ref{lang410}, $M\preceq_1 M'$ implies regularity), we can extended simultaneously $\sigma_g:D\to D$ and $\sigma_g':M'\to M'$ to an element of $\aut_{\mathcal{L}}(\mathfrak{C})$, which for simplicity we denote by the symbol ``$\sigma_g'$". Moreover, regularity of $M\subseteq M'$ implies that $M'\cap D=M$ and so $(M')^G\cap D=M^G$.

We introduce a new language, $\mathcal{L}^D:=\mathcal{L}^G\cup\lbrace M, M^G\rbrace$. Consider the following extension of $\mathcal{L}^D$-structures:
$$(D,\bar{\sigma}',M,M^G)\subseteq \big(\mathfrak{C},\bar{\sigma}',M',(M')^G\big).$$
Because $G$ is finitely generated, $M^G$ as a predicate is definable by a $\mathcal{L}^G\cup\lbrace M\rbrace$-formula, and we could skip explicite use $M^G$ and $(M')^G$ in the above extension and forget about symbol $M^G$ in the definition of language $\mathcal{L}^D$, but we want to keep things more transparent.
\
\\
\textbf{Claim:} The type $\qftp_{\mathcal{L}^D}^{\mathfrak{C}}(M'/M)$ is consistent with $\theo_{\mathcal{L}^D}(D)$.
\\
\textit{Proof of the claim:} The claim follows from the definability of $M^G$ and $(M')^G$ in $\mathcal{L}^G$-structures $(M,\bar{\sigma})$ and $(M',\bar{\sigma}')$, and from $(M,\bar{\sigma})\preceq_1(M',\bar{\sigma}')$.

Let $(D,\bar{\sigma},M,M^G)\preceq(D_1,\bar{\sigma}_1, M_1, M^G_1)$ be such that $D_1\preceq\mathfrak{C}$ and $(D_1,\bar{\sigma}_1, M_1, M^G_1)$ realizes $\qftp_{\mathcal{L}^D}^{\mathfrak{C}}(M'/M)$. Moreover, let $M''\subseteq D_1$ be a realization of $\qftp_{\mathcal{L}^D}^{\mathfrak{C}}(M'/M)$:
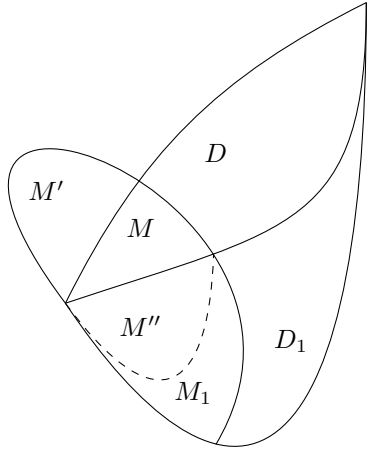
\begin{flushleft}
\hspace{10mm}
\begin{tikzpicture}
\draw  (0,0) .. controls (1,2) and (2,3) .. (4,4);
\draw  (0,0) .. controls (3,1) and (4,1) .. (4,4);
\draw  (0,0) .. controls (2,-3) and (4,-3) .. (4,4);
\draw  (0,0) .. controls (-3,4) and (4,1.5) .. (2,-1.88);
\draw[dashed]  (0,0) .. controls (1.1,-1.5) and (1.9,-1.4) .. (1.97,0.66);
\node (a) at (1,1) {$M$};
\node (a) at (-0.25,1.5) {$M'$};
\node (a) at (2,2) {$D$};
\node (a) at (3,-0.5) {$D_1$};
\node (a) at (1,-0.3) {$M''$};
\node (a) at (1.7,-1.2) {$M_1$};
\end{tikzpicture}
\end{flushleft}
There exists an $\mathcal{L}^G$-isomorphism (over $M$)
$$h:(M,\bar{\sigma})\to(M'',\bar{\sigma}_1),$$
which is a restriction of some $\hat{h}\in\aut_{\mathcal{L}}(\mathfrak{C})$, for simplicity we denote $\hat{h}$ by ``$h$". We have
$$\big(h^{-1}(D), \bar{\sigma}^{h^{-1}},M,M^G\big)\preceq \Big(h^{-1}(D_1),\bar{\sigma}_1^{h^{-1}},h^{-1}(M_1),\big(h^{-1}(M_1)\big)^G\Big),$$
$M'=h^{-1}(M'')\subseteq h^{-1}(M_1)$ and, since $h:M'\to M''$ is an $\mathcal{L}^G$-isomorphism, $\sigma_{1,g}^{h^{-1}}$ extends $\sigma_g'$ for each $g\in G$. Observe that $N\subseteq M'\subseteq h^{-1}(M_1)$ and $N^H\subseteq (M')^G\subseteq \big( h^{-1}(M_1)\big)^G$.

To this point we obtained
$$D\preceq D_2\preceq \mathfrak{C},\quad (D,M)\preceq (D_2,M_2),\quad (D,M^G)\preceq (D_2,M_2^G),$$
$$(M,\bar{\sigma})\preceq (M_2,\bar{\sigma}_2), \quad N\subseteq M_2, \quad N^H\subseteq M_2^G,$$
where $D_2:=h^{-1}(D_1)$, $M_2:=h^{-1}(M_1)$ and $\bar{\sigma}_2:=\bar{\sigma}_1^{h^{-1}}$.

By Lemma 3.55 in \cite{Hoff3}, $M^G$ is bounded. Proposition 2.5 in \cite{PilPol} (and its proof) implies that the restriction map
$$R:\mathcal{G}(M_2^G)\to
\mathcal{G}(/M^G)$$
is an isomorphism. Consider one more restriction map:
$$r:\mathcal{G}(M_2\cap \acl_{\mathcal{L}}^{\mathfrak{C}}(M_2^G))\to \mathcal{G}(M\cap \acl_{\mathcal{L}}^{\mathfrak{C}}(M^G)).$$
It is not hard to see that $M\cap\acl_{\mathcal{L}}^{\mathfrak{C}}(M^G)\subseteq M_2\cap\acl_{\mathcal{L}}^{\mathfrak{C}}(M_2^G)$ is regular. Thus, Lemma \ref{regular.Galois} implies that 
the map $r$ is onto. By Proposition 2.5 in \cite{PilPol}, it follows
$$\acl_{\mathcal{L}}^{\mathfrak{C}}(M_2^G)=\dcl_{\mathcal{L}}^{\mathfrak{C}}\big(\acl_{\mathcal{L}}^{\mathfrak{C}}(M^G), M_2^G\big)=
\dcl_{\mathcal{L}}^{\mathfrak{C}}\big(\acl_{\mathcal{L}}^{\mathfrak{C}}(M^G),
 M_2\cap\acl_{\mathcal{L}}^{\mathfrak{C}}(M_2^G)\big),$$
hence the map $r$ is injective.

We have the following diagram
$$\xymatrixcolsep{4pc}\xymatrix{
\mathcal{G}(M_2\cap \acl_{\mathcal{L}}^{\mathfrak{C}}(M_2^G)) \ar[r]^{r} \ar@{^{(}->}[d]_{\subseteq}&
\mathcal{G}(M\cap \acl_{\mathcal{L}}^{\mathfrak{C}}(M^G)) \ar@{^{(}->}[d]_{\subseteq}  \\
\mathcal{G}(M_2^G) \ar[r]^{R} \ar@{>>}[d]_{|}&
\mathcal{G}(M^G)  \ar@{>>}[d]_{|}\\
\aut_{\mathcal{L}}(M_2\cap \acl_{\mathcal{L}}^{\mathfrak{C}}(M_2^G)/M_2^G) \ar@{-->}[r]^{\exists ! R'} &
\aut_{\mathcal{L}}(M\cap \acl_{\mathcal{L}}^{\mathfrak{C}}(M^G)/M^G)
}$$
and $R'$ is an isomorphism. 
Therefore we will show that $H\in\imG(\mathcal{A})$ if we only can show that
$$H\in\imG\Big( \aut_{\mathcal{L}}(M_2\cap \acl_{\mathcal{L}}^{\mathfrak{C}}(M_2^G)/M_2^G) \Big).$$

Since $N\subseteq \acl_{\mathcal{L}}^{\mathfrak{C}}(N^H)\subseteq\acl_{\mathcal{L}}^{\mathfrak{C}}(M_2^G)$, $N\subseteq M_2$ and $N^H\subseteq M_2^G$ we have the following restriction map
$$\pi':\aut_{\mathcal{L}}(M_2\cap \acl_{\mathcal{L}}^{\mathfrak{C}}(M_2^G)/M_2^G) \to \aut_{\mathcal{L}}(N/N^H)= H.$$
We are done if we prove that $\pi'$ is onto.
Assume that $h\in\aut_{\mathcal{L}}(N/N^H)$. We have $h=\pi(g)$ for some $g\in G$. Thus
$$h=\pi(g)=\sigma_g'|_N=\sigma_{2,g}|_N=\big(\sigma_{2,g}|_{M_2\cap \acl_{\mathcal{L}}^{\mathfrak{C}}(M_2^G)}\big)|_N=\pi'\big(\sigma_{2,g}|_{M_2\cap \acl_{\mathcal{L}}^{\mathfrak{C}}(M_2^G)}\big).$$
\end{proof}

The following corollary is related to Corollary \ref{pro.embed} (since $G$ is isomorphic to a Galois group of some Galois extension) and to Corollary \ref{pro.pro.embed} (since $P$ in Corollary \ref{pro.embed7} $P$ is PAC). The devil's in the detail: in Corollary \ref{pro.pro.embed} we assume that $T$ is $\omega$-stable, but get the absolute Galois group of PAC structure, in Corollary \ref{pro.embed7} we assume that $G$ is finitely generated, but get only the Galois group of some Galois extension (similarly as in Corollary \ref{pro.embed}) of a PAC structure.

\begin{cor}\label{pro.embed7}
For every finitely generated profinite group $G$, there exist
a bounded PAC substructure $P$ of $\mathfrak{C}$
and a Galois extension $P\subseteq N$, such that
$$\aut_{\mathcal{L}}(N/P)\cong G.$$
\end{cor}

\begin{proof}
By Proposition \ref{profinite.completion} and Theorem 1.1 in \cite{NikoSega}.
\end{proof}

Now, we will state a relation between some Galois groups of $(M,\bar{\sigma})$ and group $G$, which was partially established in Corollary 3.47 in \cite{Hoff3}. To do this, definition of a \emph{Frattini cover} is needed.

\begin{definition}\label{frattini0}
Let $H,H'$ be profinite groups and $\pi: H\to H'$ be a continuous epimorphism.
The mapping $\pi$ is called a \emph{Frattini cover} if for each closed subgroup $H_0$ of $H$, the condition $\pi(H_0)=H'$ implies that $H_0=H$.
\end{definition}

\begin{cor}
The restriction map
$$\Xi:\mathcal{G}(M^G)\to \aut_{\mathcal{L}}(M\cap\acl_{\mathcal{L}}^{\mathfrak{C}}(M^G)/M^G)$$
is a Frattini cover.
\end{cor}

A Frattini cover is \emph{universal} if its domain is a projective profinite group (and then it is the smallest projective cover, the universal Frattini cover of a group $H$ will be denoted by $\fratt(H)\to H$). By Proposition 3.52 and Proposition 3.57 in \cite{Hoff3}, we know that $M^G$ and $M$ are PAC as substructures of $\mathfrak{C}$. Hence Theorem \ref{PAC.proj} shows that the map $\Xi$ is in fact universal Frattini cover. 

\begin{cor}
The restriction map $\Xi$ is the universal Frattini cover.
\end{cor}

\noindent
Moreover, Proposition \ref{profinite.completion} allows us 
to place $\hat{G}$ in the following short exact sequence
$$\mathcal{G}\big(M\cap\acl_{\mathcal{L}}^{\mathfrak{C}}(M^G)\big) \to\mathcal{G}\big(M^G\big)\to \aut_{\mathcal{L}}\big(M\cap\acl_{\mathcal{L}}^{\mathfrak{C}}(M^G)/M^G\big)\cong\hat{G}$$
We conclude:

\begin{cor}\label{cor.Galois.descr}
It follows that
\begin{enumerate}
\item
$\mathcal{G}(M^G)\cong\fratt(\hat{G})$,

\item
$\mathcal{G}\big(M\cap\acl_{\mathcal{L}}^{\mathfrak{C}}(M^G)\big)\cong\ker\Big(\fratt\big(\hat{G}\big)\to \hat{G}\Big)$.
\end{enumerate}
\end{cor}

The above corollary was known in the case of fields (\cite{sjogren}, \cite{nacfa}). Actually \cite[Theorem 6.]{sjogren} states even more:
$$\mathcal{G}(K)\cong\ker\Big(\fratt\big(\hat{G}\big)\to \hat{G}\Big),$$
where $(K,\bar{\sigma})$ is an existentially closed field with an action of group $G$. Unfortunately, the proof of \cite[Theorem 6.]{sjogren} is not correct and Theorem 6. in \cite{sjogren}, which seems to be a very reasonable statement, can not be considered as already proven.

Now, we will discuss some aspects of the proof of \cite[Theorem 6.]{sjogren} in our stable context.
Obviously, we need to examine the following restriction map:
$$\Theta:\mathcal{G}(M) \to\mathcal{G}\big(M\cap\acl_{\mathcal{L}}^{\mathfrak{C}}(M^G)\big).$$
Since $M\cap\acl_{\mathcal{L}}^{\mathfrak{C}}(M^G)\subseteq M$ is regular, Lemma \ref{regular.Galois} assures us that $\Theta$ is onto. The map $\Theta$ is one-to-one if and only if
\begin{equation}\tag{$\ast$}
\acl_{\mathcal{L}}^{\mathfrak{C}}(M)=\dcl_{\mathcal{L}}^{\mathfrak{C}}\big(M,\acl_{\mathcal{L}}^{\mathfrak{C}}(M^G) \big).
\end{equation}
The incorrect proof of \cite[Theorem 6.]{sjogren} uses the fact that $M\cap\acl_{\mathcal{L}}^{\mathfrak{C}}(M^G)$ is PAC, which is true in the case of fields, but there are no reasons for that in the general framework (since ACF is $\omega$-stable, there are only finitely many non-forking extensions of a type as required in Lemma \ref{alg.ext.PAC}, but this is not the case as follows from Remark \ref{counterexamples.alg.PAC}). Therefore we think about a slight modification of the statement of Theorem 6. in \cite{sjogren}. Mainly:

\begin{conj}\label{conjecture1}
The map $\Theta:\mathcal{G}(M) \to\mathcal{G}\big(M\cap\acl_{\mathcal{L}}^{\mathfrak{C}}(M^G)\big)$ is the universal Frattini cover.
\end{conj}
Note that Conjecture \ref{conjecture1} implies \cite[Theorem 6.]{sjogren}:

\begin{remark}
If the map $\Theta:\mathcal{G}(M) \to\mathcal{G}\big(M\cap\acl_{\mathcal{L}}^{\mathfrak{C}}(M^G)\big)$ is the universal Frattini cover and $\mathcal{G}\big(M\cap\acl_{\mathcal{L}}^{\mathfrak{C}}(M^G)\big)$ is projective then
$\mathcal{G}(M) \cong\mathcal{G}\big(M\cap\acl_{\mathcal{L}}^{\mathfrak{C}}(M^G)\big)$.
\end{remark}

Algebraic structure and model-theoretic structure of a PAC field are controlled by its absolute Galois group (e.g. Theorem 20.3.3 in \cite{FrJa}). 
The same remains true for arbitrary PAC substructures (embedded in an ambient stable monster, see \cite{DHL1}).
Conjecture \ref{conjecture1} gives us a way to produce PAC substructures of a monster model of our chosen stable theory $T$, which absolute Galois groups can be ``calculated" (as the kernel of the universal Frattini cover of the kernel of the universal Frattini cover of the profinite completion of a finitely generated group $G$).

In particular, since we start with $G$ being finitely generated, we see that 
$\hat{G}=G$ and that $\mathcal{G}(M^G)\cong\fratt(G)$ is finitely generated (by Lemma 22.6.2 in \cite{FrJa}). 
Therefore $M^G$ is bounded PAC substructure and one could expect that $M^G$ is simple (as in \cite{Polkowska}). However to use \cite{Polkowska}, one needs to verify whether being PAC is a first order property (Definition 2.7 and Section 3. in \cite{Polkowska}) or to show that the class of existentially closed substructures with $G$-action is elementary (Theorem 4.40 in \cite{Hoff3}).

In the case of finite $G$, we have that $\hat{G}=G$ and
$M\subseteq\acl_{\mathcal{L}}^{\mathfrak{C}}(M^G)$ (by point (3) in Lemma \ref{N_Galois}). Hence
$$\mathcal{G}(M)=\mathcal{G}(M\cap\acl_{\mathcal{L}}^{\mathfrak{C}}(M^G))\cong
\ker\Big(\fratt(G)\to G\Big)$$
i.e. we obtain PAC structure $M$ which absolute Galois group $\mathcal{G}(M)$ is known (compare to main results of \cite{DHL1}).

\bibliographystyle{plain}
\bibliography{1nacfa2}

\end{document}